\documentclass[ejs]{imsart}

\RequirePackage{amsthm,amsmath,amsfonts,amssymb}
\RequirePackage[numbers]{natbib}
\RequirePackage[colorlinks,citecolor=blue,urlcolor=blue]{hyperref}
\RequirePackage{graphicx}
\graphicspath{{images/}}
\RequirePackage{float}
\RequirePackage{caption}
\RequirePackage{subcaption}
\RequirePackage{linguex}
\RequirePackage{kotex}
\RequirePackage{mathrsfs}
\RequirePackage{enumitem}
\RequirePackage{mathabx}
\RequirePackage[dvipsnames]{xcolor}
\definecolor{DDbule}{RGB}{45, 72, 207}
\definecolor{GreenLine}{RGB}{0, 128, 0}
\definecolor{RedLine}{RGB}{228, 26, 28}

\doi{10.1214/154957804100000000}
\pubyear{0000}
\volume{0}
\firstpage{1}
\lastpage{1}

\startlocaldefs

\newcommand{\N}{\mathbb{N}}
\newcommand{\Normal}{\textbf{N}}  
\newcommand{\Z}{\mathbb{Z}}
\newcommand{\R}{\mathbb{R}}
\renewcommand{\I}{\mathbb{I}}

\newcommand{\SgE}{\mathcal{E}}
\renewcommand{\P}{\mathbb{P}}
\newcommand{\E}{\mathbb{E}}

\newcommand{\atom}{\boldsymbol{\alpha}}
\newcommand{\chain}{\textbf{X}}

\newcommand{\block}{\mathcal{B}}
\newcommand{\convergence}[1]{\xrightarrow{#1}}
\newcommand{\cadlag}{\mathscr{D}}
\newcommand{\cadlaginf}{\cadlag_{[0,+\infty)}}

\newcommand{\Var}{\operatorname{Var}}
\newcommand{\numreg}[1][n]{T\left({#1}\right)}
\newcommand{\irreducibilityMeasure}{\psi}
\renewcommand{\leq}{\leqslant}
\renewcommand{\geq}{\geqslant}

\newcommand{\indicator}[1]{\I\left\{#1\right\}}

\newcommand{\bootsblock}[0]{\block^{\ast}}
\newcommand{\fblock}[1][j]{f(\block_{#1})}
\newcommand{\fbootsblock}[1][j,\numreg]{f(\block^{\ast}_{#1})}
\newcommand{\floor}[1]{\lfloor #1 \rfloor}
\newcommand{\cadlagunit}[0]{\cadlag\left( [0,1] \right)}
\newcommand{\as}[0]{\convergence{\text{a.s.}}}
\newcommand{\bootsnumreg}[1][n]{T^{\ast}(n,\numreg[{#1}])}
\newcommand{\data}[1][n]{\textbf{X}^{({#1})}}

\renewcommand{\epsilon}{\varepsilon}
\renewcommand{\hat}{\widehat}

\theoremstyle{plain}
\newtheorem{theorem}{Theorem}[section]

\newtheorem{lemma}{Lemma}[section]

\newtheorem{corollary}{Corollary}[section]
\newenvironment{enumeration}{\begin{enumerate}[label=\roman*)]}{\end{enumerate}}

\theoremstyle{remark}
\newtheorem{remark}{Remark}[section]

\endlocaldefs

\begin{document}

\begin{frontmatter}

    \title{Regenerative bootstrap for \(\beta\)-null recurrent Markov chains}
    \runtitle{Regenerative bootstrap for \(\beta\)-null recurrent Markov chains}

    \author[A,B]{\fnms{Carlos} \snm{Fernández}\corref{}\ead[label=e1]{fernandez@telecom-paris.fr}}
    \address[A]{LTCI, Telecom Paris, Institut Polytechnique de Paris\\
        \printead{e1}}
    \address[B]{MODAL'X, UMR CNRS 9023, Universit\'e Paris Nanterre}
    \runauthor{C. Fernández}

    \begin{abstract}
        Two regeneration-based bootstrap methods, namely, the
        \textit{Regeneration based-bootstrap} \cite{AthreyaFuh1992, Somnat-1993} and the
        \textit{Regenerative Block bootstrap} \cite{Bertail2006} are shown to be valid
        for the problem of estimating the integral of a function with respect to the invariant measure
        in a $\beta$-null recurrent Markov chain with an accessible atom.
        An extension of the Central Limit Theorem for randomly indexed sequences
        is also presented.
    \end{abstract}




    \tableofcontents

\end{frontmatter}


\begin{acks}[Acknowledgments]
    This research has been conducted as part of the project Labex MME-DII (ANR11-LBX-0023-01).
\end{acks}

In \cite{efron1979}, Bradley Efron introduced the Bootstrap as a way to overcome some
limitations of classical methods that often relied on strong assumptions about the data’s
underlying distribution or the model’s form. Since then, these techniques, first studied
in the i.i.d. case, have been developed and extended to time-series (see
\cite{lahiri2003resampling} for an extensive survey of methods) and applied to a wide
range of problems in various fields such as signal processing
\cite{Zoubir1998,zoubirIskander2004}, soil science \cite{XiaoBing2023} and econometrics
\cite{MacKinnon2006,Horowitz2019}. These methods are easy to implement with modern
computing power and can provide more accurate and reliable inferences than traditional
methods in many situations.

Although originally designed for i.i.d. sampling, there has been significant interest in
adapting the bootstrap to situations where the data is dependent. Several resampling
methods have been proposed for time series data: these include the autoregressive-sieve
bootstrap \cite{KreissPaparoditisPolitis2011}, block bootstrap \cite{Kunsch1989} ,
circular bootstrap \cite{politis1991circular}, the stationary bootstrap
\cite{PolitisRomano1994}, continuous-path block bootstrap \cite{PaparoditisPolitis2000},
tapered block bootstrap \cite{PaparoditisPolitis2001a}, frequency-domain bootstrap
\cite{Paparoditis2002,KreissPaparoditis2003}, and local bootstrap
\cite{PaparoditisPolitis2002a}. For detailed reviews and comparisons of these methods see
\cite{FrankeKreissMammen2002,KreissPaparoditis2011,KreissLahiri2012,Cavaliere2015} and
the references therein.

In the Markovian case, numerous approaches have been developed and examined. In
\cite{Kulperger1989}, the authors proposed a block resampling scheme that consists in
resampling from a nonparametric estimate of the one-step transition matrix of a finite
state Markov chain. This method was extended to the countable case in
\cite{AthreyaFuh1992}. Extensions of this method have been proposed for the case where
the state space is Euclidean, as seen in \cite{Rajarshi1990},
\cite{PaparoditisPolitis2001,PaparoditisPolitis2002} and \cite{Horowitz2003}. The general
concept is to estimate the marginal distribution and the transition probability function
using a nonparametric function estimation technique and then resample from those
estimates. For a detailed explanation of this approach, refer to Section 4 in
\cite{KreissLahiri2012}.

A completely new approach to this problem was introduced in \cite{AthreyaFuh1992}.
Instead of using estimated transition probabilities, they exploit the regeneration
properties of a Markov chain when an accessible atom is visited infinitely often. The
main idea underlying this method consists in dividing the chain into a random number of
i.i.d. regeneration blocks and then resampling the same number of regeneration blocks.
This method, named \textit{Regeneration based bootstrap}, was proved to be valid for
finite state atomic chains in \cite{AthreyaFuh1992}, and it was extended to general
atomic positive recurrent Markov chains in \cite{Somnat-1993}.

It was pointed out in \cite{BertailClemencon2007} that the \textit{Regeneration based
bootstrap} is not second-order correct (its rate is $O_{\mathbb{P}}(n^{-1/2})$ only). To
overcome this limitation, a variation of this method, called \textit{Regenerative Block
bootstrap} (RBB), was introduced in \cite{Bertail2006}. This method consists in imitating
the renewal structure of the chain by sampling regeneration data blocks, until the length
of the reconstructed bootstrap series is larger than the length $n$ of the original data
series (notice the difference with the \textit{Regeneration based bootstrap}, where the
number of sampled blocks is equal to the number of regeneration blocks in the original
chain). It was shown in \cite{Bertail2006} that, for atomic positive recurrent Markov
chains, the RBB for estimating the integral of a function with respect to the invariant
probability, has a uniform rate of convergence of order $O_\P\left(n^{-1}\right)$ (the
same as in the i.i.d. case).

Despite all these efforts in the positive recurrent case, up to our knowledge, no
bootstrap method has been studied in the general null-recurrent scenario, although some
specific AR(1) models with unit roots have been the subject of investigations. It has
been shown that, for those AR(1) models the standard bootstrap methods (including
parametric ones) do not work \cite{Somnat1996, palm2008bootstrap} unless one works under
the null hypothesis of unit root (see \cite{ber94}). This idea can not be applied to
general null recurrent Markov chain unless one specifies the parametric or semiparametric
forms of the Markov chain. Hence, our objective in this paper is to propose a general
valid method and show that, provided that the estimators are correctly standardized, both
\textit{Regeneration based-bootstrap} and \textit{Regenerative Block bootstrap} are valid
schemes for estimating integrals with respect to the invariant measure when the Markov
chain is $\beta$-null recurrent and possesses an accessible atom. The task is challenging
because the expectation of the time of return to an atom of such Markov chains is
infinite, the bootstrap of such quantity does not work \cite{Athreya1987,Knight1989}:
indeed a necessary condition for the bootstrap to work is generally the finiteness of the
variance \cite{Csoergo2003}. Despite this fact, we will show that, by using the correct
random normalization, one can obtain a CLT for the mean and the validity of both
bootstrap regeneration methods in a null-recurrent framework.

The main difference between the methods described here and its counterparts in the
positive recurrent case is the fact that we use renormalization based on the number of
regenerations blocks (denoted by $\numreg$) instead of $n$ . Given that the number of
blocks is of order much smaller than $n$, it makes the method more restrictive. But to
our knowledge, there are no other bootstrap methods available in this null-recurrent
context (unless we consider a specific parametric model). Notice that, even standard
estimators will have rate of convergence $\sqrt{T(n)}$ with $T(n)$ or order $n^\beta$
(see \cite{Tjostheim-2001}) which of course also limits its applicability. For symmetric
random walks, $\beta=1/2$, it is known that estimators of linear functionals, kernels
estimators, volatility estimators have rate close to $n^{-1/4 }$, \cite{DelFoff2002}.

In order to make the exposition simpler, our asymptotic results will be specifically
stated for integrals with respect to the invariant measure, however, the procedures can
be applied to any statistic defined over a regeneration blocks, as long as it has finite
variance. Similarly, our results are stated for first order Markov chains, but they can
easily be extended to higher order chains by vectorization \cite[pp.
15]{markovChain2018}.

The paper is organized as follows: in section \ref{sec:ch4:markov_chains} we provide a
brief introduction to null recurrent Markov chains, making a special emphasis on atomic
ones and presenting the main results that we use throughout the paper. In subsection
\ref{sec:ch4:beta_null_recurrent} we present an extension of the Central Limit Theorem
for randomly indexed sequences (Lemma \ref{clt_general_iid}). Section \ref{sec:ch4:rbb}
is dedicated to the \textit{Regenerative Block bootstrap} in $\beta$-null recurrent
Markov chains, while Section \ref{sec:ch4:regeneration_bb} is devoted to the
\textit{Regeneration based-bootstrap}. In section \ref{sec:ch4:simulations} we have
performed two simulation studies to show the behavior of both algorithms in practice.
Section \ref{sec:conclusions} contains a few concluding remarks. The technical proofs are
postponed to Section \ref{sec:ch4:proofs}.

\section{A short introduction to null-recurrent Markov chains}\label{sec:ch4:markov_chains}

In this section, we introduce some notation and review some important concepts from
Markov chain theory that will be used throughout the paper. For more details, please
refer to \cite{Meyn2009,markovChain2018}.

\subsection{Notation and definitions}

Consider an homogeneous Markov chain \(\chain=\{X_0,X_1,\ldots\}\) on a countably
generated state space \((E,\SgE)\), with transition kernel \(P\) and initial probability
distribution \(\lambda\). This means that for any \(B\in\SgE\) and \(n\in\mathbb{N}\), we
have \(\mathcal{L}\left(X_0\right)=\lambda\) and
\begin{equation*}
    \P(X_{n+1}\in B\mid X_{0},\ldots ,X_{n})=P(X_{n},B)\quad\textnormal{almost surely}.
\end{equation*}

Note that the assumption of a countably generated state space is commonly used in Markov
chain theory to avoid pathological examples known as 'anormal' chains \cite{Doeblin1940}.
For more information on this topic, see \cite{Doeblin1940}, \cite{jainJamison1967}, and
\cite{Revesz2005}. An example of an 'anormal' chain can be found in \cite{Blackwell1945}.
This assumption does not significantly limit the generality of our results since most of
the time \(\SgE\) will be the borelian \(\sigma\)-algebra of \(\R^d\), which is countably
generated.

In the following, we use \(\mathbb{P}_{\lambda}\) (or \(\mathbb{P}_{x}\) for \(x\) in
\(E\)) to denote the probability measure on the underlying space such that
\(X_{0}\sim\lambda\) (or \(X_{0}=x\)). We use \(\E_{\lambda}\) to represent the
\(\mathbb{P}_{\lambda}\)-expectation (or \(\E_x\) to represent the
\(\mathbb{P}_{x}\)-expectation), and \(\indicator{\mathcal{A}}\) to represent the
indicator function of event \(\mathcal{A}\).

A homogeneous Markov chain is said to be irreducible if there exists a \(\sigma\)-finite
measure \(\phi\) on \((E,\SgE)\) such that for all \(x\in E\) and all \(A\in\SgE\) with
\(\phi(A)>0\), there exists some \(n\geq 1\) such that \(P^n(x,A)>0\). In this case,
there exists a maximal irreducibility measure \(\irreducibilityMeasure\) with respect to
which all other irreducibility measures are absolutely continuous. If $\chain$ is
\(\irreducibilityMeasure\)-irreducible, there is \(d^{\prime}\in\mathbb{N}^{\ast}\) and
disjoints sets \(D_{1},\ldots,D_{d^{\prime}}\) \(D_{d^{\prime}+1}=D_{1}\) weighted by
\(\psi\) such that \(\psi(E\backslash\cup_{1\leq i\leq d^{\prime}}D_{i})=0\) and
\(\forall x\in D_{i}, P(x,D_{i+1})=1.\) The g.c.d. \(d\) of such integers is called the
\textit{period} of the chain. $\chain$ is said to be \textit{aperiodic} if \(d=1\).

Thorough this paper, we assume that the Markov chains under consideration are
homogeneous, aperiodic, and irreducible with maximal irreducibility measure
\(\irreducibilityMeasure\).

An irreducible chain possesses an accessible atom, if there is a set \(\atom\in\SgE\)
such that for all \(x,y\) in \(\atom\):
\(P\left(x,\bullet\right)=P\left(y,\bullet\right)\) and
\(\irreducibilityMeasure(\atom)>0\). For instance, when a chain can take a countable
number of values, any single point visited by the chain is an atom. Denote by
\(\sigma_{\atom}\) and \(\tau_{\atom}\), respectively, the times of first visit and first
return of the chain to ${\atom}$, i.e. \(\tau_{\atom}=\operatorname{inf}\left\{ n\geq 1:
X_n\in \atom \right\}\) and \(\sigma_{\atom}=\operatorname{inf}\left\{ n\geq 0: X_n\in
\atom \right\}\). The subsequent visit and return times \(\sigma_{\atom},
\tau_{\atom}\left(k\right)\), \(k\geq 1\) are defined inductively as follows:
\begin{align}
    \tau_{\atom}\left(1\right) = \tau_{\atom}\quad     & ,\quad \tau_{\atom}\left(k\right) = \operatorname{min}\left\{ n>\tau_{\atom}\left( k-1 \right):X_n\in \atom\right\},     \\
    \sigma_{\atom}\left(1\right) = \sigma_{\atom}\quad & ,\quad \sigma_{\atom}\left(k\right) = \operatorname{min}\left\{ n>\sigma_{\atom}\left( k-1 \right):X_n\in \atom\right\}.
\end{align}

We use \(T_n(A)\) to represent the random variable that counts the number of times the
chain visits the set \(A\) up to time \(n\), i.e., \(T_n(A)=\sum_{t=0}^{n}{\I\{X_t\in
A\}}\). Similarly, we use \(T_{\infty}(A)\) to represent the total number of visits of
chain \(\chain\) to \(A\).

An atom \(\atom\) is called \textit{recurrent} if \(\E_x T_{\infty}(\atom)=+\infty\) for
all \(x\in \atom\); otherwise, it is called \textit{transient}. A notable property of
atomic chains is that all accessible atoms are either all recurrent or all transient.
Therefore, we say that an atomic chain is recurrent if one (and thus all) of its
accessible atoms is recurrent. If \(\chain\) is aperiodic, recurrent and possesses an
accessible atom, then the probability of returning infinitely often to the atom \(\atom\)
is equal to one, no matter the starting point, i.e.
\begin{equation*}
    \P_{x}\bigl(T_{\infty}(\atom)=\infty \bigr)=1\quad\textnormal\; \forall x\in E.
\end{equation*}

Denote by \(\P_{\atom}\) and \(\E_{\atom}\) the probability and the expectation
conditionally to \(X_0\in \atom\).

A fundamental tool for understanding the long-term behavior of Markov chains is the
existence of invariant measures, that is, a measure \(\pi\) such that
\begin{equation*}
    \pi \left( A \right) = \int {P\left( {x,A} \right)d\pi \left( x \right)}\quad\forall A\in\SgE.
\end{equation*}

Every irreducible and recurrent Markov chain admits a unique (up to a multiplicative
constant) invariant measure \cite[Theorem 10.4.9]{Meyn2009}. In the atomic case, the
invariant measure is just the occupation measure over the first block $\block_1=\left(
X_{\tau_{\atom}\left( 1 \right)+1},\ldots,X_{\tau_{\atom}\left( 2 \right)} \right)$
\cite[Theorem 6.4.2]{markovChain2018}, i.e.
\begin{equation}\label{eq:ch3:invariantMeasure}
    \pi_{\atom}\left(A\right) = \mathbb{E}_{\atom} \left( \sum_{j=1}^{\tau_{\atom}} \indicator{X_{j}\in A}  \right),
    \quad \forall A\in\SgE.
\end{equation}

An irreducible Markov chain is \textit{positive recurrent} if its invariant measure is
finite. When the invariant measure is just $\sigma$-finite, then the chain is called
\textit{null recurrent}. From \eqref{eq:ch3:invariantMeasure}, it is clear that an atomic
Markov chain is positive recurrent if and only if $\E_{\atom}{\tau_{\atom}<+\infty}$, and
in this case, the measure defined by ${\pi_{\atom}}/{\E_{\atom}\tau_{\atom}}$ is an
invariant probability for the chain. The existence of this invariant probability makes
the theory of positive recurrent Markov chains, very similar to the i.i.d. case
\cite[Chapter 17]{Meyn2009}.

Conversely, dealing with null recurrent chains is considerably more challenging, and a
comprehensive theory of non-parametric estimation for this type of chain does not exist.
To address this issue, Karlsen and Tjøstheim introduced in \cite{Tjostheim-2001} a
regularity condition for the tail behavior of the distribution of $\tau_{\atom}$ that
renders the problem more tractable. Specifically, denote by $\Gamma$ the gamma
function\footnote{The \(\Gamma\) function is defined as
\(\Gamma(x)=\int_{0}^{+\infty}{t^{p-1}\exp({-t})}\,dt\).}, then, a chain is referred to
as $\beta$-null recurrent if there is a constant $\beta\in(0,1)$ and a slowly varying
function\footnote{A measurable and positive function $L$ is said to be \textit{slowly
varying at $+\infty$} if it is defined in $[a,+\infty)$ for some $a\geq 0$, and satisfies
$\lim_{x\to +\infty}{{L\left( xt \right)}/{L\left( x \right)}}=1$ for all $t\geq a$. For
a detailed discussion on these types of functions, refer to \cite{Bingham1987}. } $L$
such that
\begin{equation}\label{tailCondition}
    \P_{\atom}\left(\tau_{\atom} > n\right)\sim\frac{1}{\Gamma(1-\beta)n^\beta L(n)}.
\end{equation}

The number $\beta$, also known as the \textit{regularity index} (see
\cite{Chen1999,Chen2000}) satisfies
\begin{equation*}
    \beta = \sup\left\{p>0:\E_{\atom}\left(\tau_{\atom}^p\right)<+\infty\right\}.
\end{equation*}

Some of the most well-known examples of $\beta$-null recurrent Markov chains are the
random walks in $\R$, which are $1/2$-null recurrent \cite{Kallianpur1954}, the Bessel
random walks \cite{Deconinch2009, Alexander2011} and some types of threshold
autoregressive (TAR) \cite{Gao2013} and vector autoregressive processes (VAR)
\cite{Myklebust-2012}. $\beta$-null recurrent Markov chains appear naturally in many
fields of statistics and probability for instance for studying population dynamics,
statistical mechanics or the study of Polymer.

\subsection{Renewal properties and Block decomposition}\label{sec:ch4:preliminary_remarks}

The \textit{strong Markov property} implies that the sample paths of an atomic Markov
chain can be partitioned into independent blocks of random length corresponding to
consecutive visits to $\atom$, given by:
\begin{align*}
    \block_0 & =\left( X_0,X_1,\ldots,X_{\tau_{\atom}\left( 1 \right)} \right)                              \\
    \block_1 & =\left( X_{\tau_{\atom}\left( 1 \right)+1},\ldots,X_{\tau_{\atom}\left( 2 \right)} \right)   \\
             & \ldots                                                                                       \\
    \block_n & =\left( X_{\tau_{\atom}\left( n \right)+1},\ldots,X_{\tau_{\atom}\left( n+1 \right)} \right) \\
             & \ldots
\end{align*}

Note that the distribution of $\block_0$ depends on the initial measure, and thus it does
not have the same distribution as $\block_j$ for $j\geq 1$. The sequence
$\{\tau_{\atom}(j)\}_{j\geq 1}$ defines successive times at which the chain forgets its
past, which are called \textit{regeneration times}. Similarly, the sequence of i.i.d.
blocks $\{\block_j\}_{j\geq 1}$ is called \textit{regeneration blocks}. As customary in
the \(\beta\)-null recurrent Markov chain literature, we will use \(\numreg\) to denote
the number of complete regeneration blocks up to time \(n\), i.e.
\(\numreg=\max\left(T_n(\atom)-1,0\right)\). We will denote by
$\ell\left(\block_i\right)$ the length of the $i$-th block, therefore,
\begin{equation}\label{eq:ch4:block_sizes}
    \ell\left(\block_j\right)=
    \begin{cases}
        \tau_{\atom}                                            & ,\quad j=0     \\
        \tau_{\atom}\left(j+1\right)-\tau_{\atom}\left(j\right) & ,\quad j\geq 1
    \end{cases}
\end{equation}

The random variable $\numreg$, and its relationship with
$\sum_{j=0}^{k}\ell\left(\block_j\right)$, is crucial in the theory we will develop in
this paper, therefore, we will state in this section its main properties in the
$\beta$-null recurrent scenario.

Assume $\chain$ is a $\beta$-null recurrent Markov chain with an accessible atom $\atom$.
By (3.27) in \cite{Tjostheim-2001}, the function $L$ in \eqref{tailCondition} can be
normalized in such a way that
\begin{equation}\label{eq:ch4:def_u_n}
    u(z)=z^\beta L(z)
\end{equation}
is a continuous function that is strictly increasing in the interval $[z_0,+\infty)$
for some $z_0\in\R_+$. Define $v(z)$ as
\begin{equation}\label{eq:ch4:def_v_n}
    v(z)=u^{(-1)}(z)=\inf\left\{s: u\left(s\right)>z\right\},
\end{equation}
then, $u\left(v(z)\right)=v\left(u(z)\right)=z$ for $z\geq z_0$.

Consider the space of càdlàg functions defined on the interval $[0,+\infty)$, denoted by
$\cadlaginf$. This space consists of the real functions that are right-continuous with
left limits and defined over $[0,+\infty)$. More precisely, a function $g\in\cadlaginf$
if and only if $g$ is right-continuous, has left limits at all points $t>0$, and
$\lim_{t\downarrow 0} g(t)=g(0)$. The space $\cadlaginf$ is equipped with the
Skorokhod\footnote{See Chapter 6 of \cite{JeanJacod2003} or Chapter 3 in
\cite{Billingsley1968} for more details about this space.} topology, making it a complete
and separable metric space. We will use $\convergence{\cadlaginf}$ to denote weak
convergence in this space, and $\convergence{\textnormal{fd}}$ for convergence of
finite-dimensional laws. Two stochastic processes $Y_n$, $Z_n$ in $\cadlaginf$ are said
to be \textit{equivalent} if $Y_n-Z_n$ converges weakly to the zero process. If
$Y_n\convergence{\cadlaginf}Y$ and $Y_n$ and $Z_n$ are equivalent, then
$Z_n\convergence{\cadlaginf}Y$ (see Lemma 3.31 in \cite{JeanJacod2003}).

Define the following processes
\begin{equation}\label{eq:ch4:t_n_process}
    T_n(t)=\frac{T\left(\left\lfloor nt \right\rfloor\right)}{u\left( n \right)},\quad C_n(t)=\frac{1}{v\left( n \right)}\sum_{k=0}^{\left\lfloor nt \right\rfloor}{\ell\left(\block_k\right)},
\end{equation}
and $C_n^{(-1)}(t)=\inf\{x:C_n(x)>t \}$.
The following Theorem, proved in \cite{Tjostheim-2001}, shows that these three processes converge
in $\cadlaginf$ and that $T_n$ and $C_n^{(-1)}$ are equivalent.
\begin{theorem}\label{th:ch4:convergence_process} Assume $\chain$ is a $\beta$-null recurrent atomic
    Markov chain. Then,
    \begin{enumeration}
        \item $C_n\convergence{\cadlaginf} S_{\beta}$ where $S_\beta$ is the one-sided stable Levy process
        defined by the marginal characteristics
        \[\E\left[ \exp\left( is S_\beta(t) \right) \right]=\exp\left( is^\beta t \right)\;s\in\R,t\in[0,+\infty].\]
        \item $C_n^{(-1)}$ and $T_n$ are equivalent processes and both converge in $\cadlaginf$ to the
        Mittag-Leffler process of parameter $\beta$.
    \end{enumeration}
\end{theorem}
\begin{remark}
    The Mittag-Leffler process with parameter \(\beta\) is defined as the inverse of \(S_\beta\).
    It is a strictly increasing continuous stochastic process defined as
    \begin{equation*}
        M_\beta(t)=t^\beta M_\beta\left(1\right)\quad,\quad \E \left(M^m_\beta\left(1\right)\right)=\frac{m!}{\Gamma\left(1+m\beta\right)}\;\;m\geq 0.
    \end{equation*}
\end{remark}

Theorem \ref{th:ch4:convergence_process} shows a striking difference between positive and
null recurrent Markov chains. While in the former the existence of moments for
$\ell\left(B_j\right)$ implies that $C_n$ and $T_n$ (taking $u(n)=n$) converge almost
surely respectively to $t\E_{\atom}\tau_{\atom}$ and ${t}/{\E_{\atom}\tau_{\atom}}$, and
therefore, $\numreg$ can be approximated almost surely by the deterministic quantity
$n/{\E_{\atom}\tau_{\atom}}$, in the latter, we only have weak convergence, hence
$\numreg$ can only be controlled by the deterministic quantity $u(n)$ in distribution.

\subsection{Properties of linear functionals defined on $\beta$-null recurrent
    chains}\label{sec:ch4:beta_null_recurrent}

For a measurable function \(f:E\rightarrow\mathbb{R}\), and an atomic Markov chain
$\chain$ with an accessible atom $\atom$, consider the problem of estimating
\(\pi_{\atom}(f)=\int f d\pi_{\atom}\), where $\pi_{\atom}$ is as in
\eqref{eq:ch3:invariantMeasure} and $\pi_{\atom}(f)<+\infty$. Denote by \(S_n(f)\) the
partial sums of \(f\) over the chain, that is
\begin{equation}\label{eq: ch4: partial_sums_f_chain}
    S_n(f)=\sum_{k=0}^{n}{f\left(X_k\right)}.
\end{equation}

The Ratio Limit Theorem for atomic chains \cite[Theorem 6.6.2]{markovChain2018} shows
that if $g$ is a measurable function, then, for every invariant measure $\pi$ we have
\begin{equation}\label{eq:ch4:ratio_limit_theorem}
    \frac{S_n(f)}{S_n\left(g\right)}\as\frac{\pi(f)}{\pi\left(g\right)},
\end{equation}
as long as $\pi\left(g\right)\neq 0$.

\begin{remark}\label{rm:ch4:non_convergence_null_recurrent}From \eqref{eq:ch4:ratio_limit_theorem} is
    clear that ${S_n(f)}/{\numreg}$ is a strongly consistent estimator of
    $\pi_{\atom}(f)$, and, in the positive recurrent case,
    ${S_n(f)}/{n}\as{\pi_{\atom}(f)}/{\E_{\atom}\tau_{\atom}}$. In the
    null recurrent case, however, ${S_n(f)}/{n}\as 0$ (see Corollary 6.6.3 in
    \cite{markovChain2018}) and there is no deterministic sequence $a\left(n\right)$ such
    that ${S_n(f)}/{a(n)}$ converges almost surely to a non-zero limit \cite{Chen1999}.
\end{remark}

Given that our interest in this paper is to apply the bootstrap method to the study of
$\pi_{\atom}(f)$ we need to find a series of i.i.d. random variables whose mean strongly
converges to $\pi_{\atom}(f)$. To do this, define the following random variables
\begin{equation*}
    f(\block_j)=
    \begin{cases}
        \sum\limits_{i=0}^{\tau_{\alpha}}{f\left( X_i \right)}                                                & ,\quad j=0     \\
        \sum\limits_{i=\tau_{\atom}\left( j \right)+1}^{\tau_{\alpha}\left( j+1 \right)}{f\left( X_i \right)} & ,\quad j\geq 1
    \end{cases}.
\end{equation*}

The strong Markov property implies that under \(\P_{\atom}\), the sequence
\(\{f(\block_j)\}_{j\geq 0}\) is i.i.d. Moreover, for every initial probability
\(\lambda\) such that \(\P_{\lambda}\left(\tau_{\atom}<\infty\right)=1\), the random
variables \(f(\block_j),j\geq 0\) are independent and for \(j\geq 1\) they are i.i.d.
Therefore, $S_n(f)$ can now be written as a sum of independent random variables as
follows:
\begin{equation}\label{eq: ch4: split_sum_equation}
    S_n(f)=f\left(\block_0\right)+\sum_{j=1}^{T(n)}{f\left( \block_j \right)}+\sum_{i=\tau_{\atom}\left( T(n)+1 \right)+1}^n{f\left(X_i\right)},
\end{equation}
with the convention that the sum of an empty set is 0. As customary in the $\beta$-null
recurrent literature, we will denote the last term in \eqref{eq: ch4: split_sum_equation}
by $f(\block_{(n)})$.

Equation \eqref{eq:ch3:invariantMeasure} indicates that
\(\E_{\atom}{f(\mathcal{B}_{j})}=\pi_{\atom}(f)\) for \(j\geq 1\), hence, if we assume
that $\pi_{\atom}(|f|)<+\infty$, the Law of Large Numbers for randomly indexed sequences
\cite[Theorem 8.2, pp 302]{Gut2013} shows that
\begin{equation}\label{eq:ch4:lln}
    \frac{1}{T(n)}\sum_{j=1}^{T(n)}{f(\mathcal{B}_{j})}\as\pi_{\atom}(f).
\end{equation}

\begin{remark}\label{rm:ch4:negligible_blocks}

    The recurrence of the chain implies that \(\numreg\to\infty\) almost surely,
    therefore \(f(\block_0)/\numreg\) and \(f(\block_{(n)})/\numreg\) converge to 0
    almost surely (see Lemma 1 in \cite{Athreya-2016}). This allows us to consider only
    the i.i.d. blocks \(f(\block_j),j\geq 1\) in our estimations.

\end{remark}


If we suppose further that \(f(\mathcal{B}_{1})\) has finite second moment, and we denote
by $\sigma^2$ the variance of \(f(\mathcal{B}_{1})\), then
\begin{equation}\label{eq:ch4:varianceEstimator}
    \hat{\sigma}_n^2=\frac{1}{T(n)}\sum_{j=1}^{T(n)}{\left(\fblock-\frac{1}{T(n)}\sum_{i=1}^{T(n)}{\fblock[i]}\right)^2}\as\sigma^2.
\end{equation}

Much of the work carried out in this investigation deals with sequences indexed by the
sequence of random variables \(\numreg\). As explained at the end of Section
\ref{sec:ch4:preliminary_remarks}, this sequence, although it converges almost surely to
\(+\infty\), can not be deterministically approximated in probability, it only admits an
approximation in distribution. This creates huge problems, even for simple tasks, as to
obtaining a CLT, because CLTs for randomly indexed sequences (see \cite{Anscombe1952} for
the original formulation and Th. 17.2 in \cite{Billingsley1968} for its more general
form) require being able to control deterministically, at least in probability, the
sequence of the number of terms. The result we present below extends this CLT, replacing
the requirement of the control in probability by the existence of the limit of a
stochastic process defined in terms of the sequence of the number of terms.

\begin{lemma}[CLT for randomly indexed sequences]\label{clt_general_iid}
    Let \(X_{1},X_2\ldots\) be i.i.d. random variables such that
    \(\E(X_{1})=\mu\) and \(\Var{}{X_{1}}=\sigma^2>0\). Let \(N(n)\)
    be a sequence of integer-valued random variables. Assume there exists an unbounded
    increasing sequence of real numbers \(u_n\) such that the process
    \(N_n(t)={N(\floor{nt})}/{u_n}\) satisfies the following conditions:
    \begin{itemize}
        \item There exists a process \(S_n\) in \(\cadlaginf\) that is non-negative and
              non-decreasing for each \(n\).
        \item \(S_n\convergence{\cadlaginf} S\) where \(S\) is
              a strictly increasing non-negative process with independent
              increments, no fixed jumps, and \(S(0)\equiv 0\).
        \item \(N_n\) is equivalent to \(S_n^{(-1)}\).
    \end{itemize}
    Then, \(N_n\) converges to \(S^{(-1)}\) in \(\cadlaginf\),
    \begin{equation*}
        \frac{\sqrt{N(n)}}{\sigma}\left(\frac{1}{N(n)}\sum_{j=1}^{N(n)}\left(X_j-\mu\right)\right)\convergence{d}\Normal(0,1),
    \end{equation*}
    and $N_n(1)$ and $\frac{\sqrt{N(n)}}{\sigma}\left(\frac{1}{N(n)}\sum_{j=1}^{N(n)}\left(X_j-\mu\right)\right)$ are asymptotically independent.
\end{lemma}

\begin{corollary}\label{cor: clt_original}[Theorem 17.2 in \cite{Billingsley1968}] Suppose \(X_1,\ldots,X_n\) are i.i.d. with \(\E X_1=\mu\)
    and \(\Var{}{X_{1}}=\sigma^2\). If \(N(n)\) is a sequence of integer-valued
    random variables such that
    \begin{equation}\label{eq: cond_convergence_prob}
        \frac{N(n)}{u_n}\convergence{p}\theta,
    \end{equation}
    where \(\theta\) is a positive random variable and \(u_n\) is a sequence of positive numbers
    going to infinity, then
    \begin{equation*}
        \frac{\sqrt{N(n)}}{\sigma}\left(\frac{1}{N(n)}\sum_{j=1}^{N(n)}\left(X_j-\mu\right)\right)\convergence{d}\Normal(0,1).
    \end{equation*}
\end{corollary}

Using Lemma \ref{clt_general_iid} and Theorem \ref{th:ch4:convergence_process} we can
provide a different proof of the following Central Limit Theorem for $\beta$-null
recurrent atomic Markov chains, which was originally proved in \cite{Athreya2015}.

\begin{theorem}\label{cltLemma}

    Let $\chain$ be a $\beta$-null recurrent Markov chain, with an accessible atom
    $\atom$. For every $\pi_{\atom}$- measurable function $f$ such that
    $\sigma^2=\Var\fblock[1]$ is finite, we have the following convergence in
    distribution:
    \begin{equation}
        \frac{\sqrt{T(n)}}{\sigma}\left(\frac{1}{T(n)}\sum\limits_{j=1}^{T(n)}{f(\mathcal{B}_{j})}-\int f \,d\pi_{\atom}\right)\convergence{d}\Normal(0,1).\label{clt}
    \end{equation}
    Moreover, \({T(n)}/{n^\beta L(n)}\) converges to a
    Mittag-Leffler distribution with parameter \(\beta\) and it is asymptotically
    independent of the left-hand side of \eqref{clt}. If in addition we also have
    $\E [|f|(\block_1)^2]<+\infty$, then
    \begin{equation}
        \frac{\sqrt{T(n)}}{\sigma}\left(\frac{S_n(f)}{T(n)}-\int f \,d\pi_{\atom}\right)\convergence{d}\Normal(0,1).\label{clt_non_blocks}
    \end{equation}
\end{theorem}

The following corollary is a direct consequence of Theorem \ref{cltLemma}, equation
\eqref{eq:ch4:varianceEstimator} and Slutsky's theorem.

\begin{corollary}\label{cor:clt_studentized}
    Under the same hypothesis of Theorem \ref{cltLemma}, we have
    \begin{align*}
        \frac{\sqrt{T(n)}}{\hat{\sigma}_n}\left(\frac{1}{T(n)}\sum\limits_{j=1}^{T(n)}{f(\mathcal{B}_{j})}-\int f \,d\pi_{\atom}\right) & \convergence{d}\Normal(0,1),
    \end{align*}
    and if $\E [|f|(\block_1)^2]<+\infty$ also holds, then
    \begin{equation*}
        \frac{\sqrt{T(n)}}{\hat{\sigma}_n}\left(\frac{S_n(f)}{T(n)}-\int f \,d\pi_{\atom}\right)               \convergence{d}\Normal(0,1).
    \end{equation*}
\end{corollary}

\section{The regenerative block-bootstrap algorithm (RBB)}\label{sec:ch4:rbb}

Let \(\data=(X_{0},..., X_{n})\) be observations drawn from a \(\beta\)-null recurrent
Markov chain \(\chain\) with an \textit{a priori} known accessible atom \(\atom\). As in
the previous section, let $f$ be a $\pi_{\atom}$-integrable function such that
$f\left(\block_1\right)$ has a finite second moment. Denote by $\sigma^2$ the variance of
$f\left(\block_1\right)$.

The Regenerative block-bootstrap (RBB) method, which we explore in this section, was
initially introduced in \cite{Bertail2006} for positive recurrent Markov chains. In their
Theorem 2.1, it was shown that, in the atomic case, the RBB distribution achieves a
uniform rate of convergence of order $O_p\left(n^{-1}\right)$ for both the studentized
and unstudentized sample mean, meaning that the sup-norm between the true distribution
and its bootstrap approximation is of order $O_p\left(n^{-1}\right)$.

In this section, we show that the method is also applicable in the $\beta$-null recurrent
case, although we have not been able to obtain a rate.

\begin{remark}

    Obtaining rates of convergence, for the bootstrap, typically depends on Edgeworth
    expansions \cite{HALL1990108}. These expansions can be derived, at least formally, by
    calculating cumulants using standard techniques. In the Markovian case, the validity
    of these expansions not only depends on the cumulants of $\fblock$ but also on the
    moments of \(\tau_{\atom}\) \cite{Malinovskii1987, BertailClemencon2004}. More
    precisely, up to our knowledge, these expansions have been obtained only when
    \(\E_{\atom}\tau_{\atom}^4\) is finite \cite[Theorem 5.1]{BertailClemencon2004}.
    Developing methods to obtain Edgeworth expansions for distributions with very few
    moments is an interesting research direction, but it would involve substantial
    theoretical developments that are beyond the scope of this work.

\end{remark}

Proposition 3.1 in \cite{BertailClemencon2004} shows that for positive recurrent chains,
in the nonstationary case (when the initial law $\lambda$ is not the invariant
probability measure), the first data block $\mathcal{B}_{0}$ induces a bias of order
$O(n^{-1})$, which cannot be estimated from a single realization $\data$ of the chain
starting from $\lambda$. The last block $\mathcal{B}_{(n)}$ (which is incomplete) induces
a first-order term in the bias too. This led the authors in \cite{Bertail2006} to only
consider statistics based on the regenerative data blocks
$\block_{1},\ldots,\block_{\numreg}$.

In the $\beta$-null recurrent case, the lack of finite first moment for the block sizes
suggests that considering the non-regenerative blocks will incur in an even worse bias,
hence, as in \cite{Bertail2006}, we will only consider statistics based on the
regenerative data blocks $\block_{1},\ldots,\block_{\numreg}$.

While our asymptotic results are specifically stated for integrals with respect to the
invariant measure, the algorithm can be applied to any statistic defined over the
regeneration blocks, as long as it has finite variance.

As customary in the bootstrap literature, $\P^{\ast}(\bullet)=\P(\bullet\mid \data)$
denotes the conditional probability given $\data$. We will write $Z^{\ast}_n
\xrightarrow{d^{\ast}}_p Z$ to indicate the weak converge in probability of the bootstrap
random variables $Z_n^\ast$ to $Z$, this is, for all $x\in\R$, $\P^{\ast}(Z_n^{\ast}\leq
x)\convergence{p}\P(Z\leq x)$. See pp. 2550 in \cite{CavaliereGeorgiev2020} for more
details.

In this section, our goal is to bootstrap a general statistic $G_n$ that converges to a
parameter $\theta$. We will typically prove asymptotic results for the case where
$G_n=\sum_{j=1}^{\numreg}{\fblock[i]}/{\numreg}$. Additionally, we assume the
availability of a block-based standardization, denoted as
$\;Std_{n}=Std(\block_{1},...,\block_{\numreg})\;$. The distribution of interest is
defined as $H_n(x)=\P(Std_{n}^{-1}(G_n -\theta) \leq x)$.

The RBB procedure is performed in four steps as follows:

\begin{enumerate}
    \item Count the number of visits\ $T_n(\atom)$ to the atom $\atom$ up to time $n$,
          and divide the observed sample path\ $\data=(X_{0},\ldots,X_{n})$ into
          $T_n(\atom)+1$ blocks,\ $\block_{0}$, $\block_{1},\ldots,$
          $\block_{T_n(\atom)-1},$ $\block_{T_n(\atom)\;}^{(n)}$, corresponding to the
          pieces of the sample path between consecutive visits to the atom $\atom$.\ Drop
          the first and last (non-regenerative) blocks. Denote by $\numreg$ the number of
          remaining blocks.

    \item Draw sequentially bootstrap data blocks $\block_{1,\numreg}^{\ast},..., $
          $\block_{k,\numreg}^{\ast}$ independently from the empirical distribution
          $F_{n}=\numreg^{-1}\sum_{j=1}^{\numreg}\delta_{\block _{j}} $ of the blocks
          $\{\block_{j}\}_{1\leq j\leq \numreg}$ conditioned on $\data$, until the length
          $\ell^{\ast}(k)=\sum_{j=1}^{k}\ell(\block_{j,\numreg}^{\ast})$ of the bootstrap
          data series is larger than $n$. Let $T^{\ast}_n(\atom)=\inf\{k\geq1,$
          $\ell^{\ast}(k)>n\}$ and \(\bootsnumreg=T^{\ast}_n(\atom)-1\).

    \item From the data blocks generated in step 2, reconstruct a pseudo-trajectory of
          size $\ell^{\ast}(\bootsnumreg)\;$by binding the blocks together, that is
          \begin{equation*}
              X^{\ast(n)}=\left(\block_{1,\numreg}^{\ast},...,\block_{\bootsnumreg,\numreg}^{\ast}\right).
          \end{equation*}
          Compute the \textit{RBB statistic}$\ G_{n}^{\ast}=G_{n}(X^{\ast(n)}).$

    \item If$\;Std_{n}=S(\block_{1},...,\block_{\numreg})\;$is\ an appropriate
          standardization of the original statistic $G_{n}$, compute
          $Std_{n}^{\ast}=S(\block_{1,\numreg}^{\ast\text{ }},...,\block
          _{\bootsnumreg,\numreg}^{\ast})$.
\end{enumerate}

The \textit{RBB distribution} is then given by
\begin{equation*}
    H_{RBB}(x)=\P^{\ast}\Bigl(Std_{n}^{\ast-1}\left(G_{n}^{\ast}-G_{n}\right)\leq x\Bigr).
\end{equation*}
One purpose of the next paragraphs is to show that if we choose a correct standardization $Std_n$, then we can obtain that $H_{RBB}(x)-H_n(x)\convergence{p}0$ uniformly in $x$.
Our main asymptotic result, in the case of integrals concerning the invariant measure, is
the following.

\begin{theorem}[Validity of the RBB]\label{th:ch3:convergence_rbb}
    Let $\chain$ be a $\beta$-null recurrent Markov chain with an accessible atom $\atom$, and let $f$
    be a $\pi_{\atom}$-integrable function such that $\E [f(\block_1)^2]<+\infty$.
    Define
    \begin{equation*}
        \hat{\sigma}^2_{\numreg} = \frac{1}{\numreg}\sum\limits_{j = 1}^{\numreg} {{{\left( {\fblock - \frac{1}{\numreg}\sum\limits_{i = 1}^{\numreg} {\fblock[i]} } \right)}^2}}\textnormal{ and }\hat{\mu}_{\numreg}=\frac{1}{{\numreg}}\sum\limits_{i = 1}^{\numreg} {\fblock[i]}.
    \end{equation*}
    Then we have,
    \begin{equation*}
        \frac{\sqrt {\bootsnumreg}}{{\hat{\sigma}_{\numreg}}} \left( \frac{1}{\bootsnumreg} {{\sum\limits_{j = 1}^{\bootsnumreg} {\left( {\fbootsblock - \hat{\mu}_{\numreg} } \right)} }} \right) \convergence{d^\ast}_p  \Normal\left( {0,1} \right).
    \end{equation*}
\end{theorem}

This theorem yields that the bootstrap distribution of the standardized sum has
asymptotically the same distribution as the statistics
${\sum_{j=1}^{T(n)}f(\mathcal{B}_{j})}/{T(n)}$ estimating $\int f d\pi_{\atom}$. The
proof of this result is non-trivial and totally non-standard: it starts by constructing a
space, via Skorokhod-Dudley-Wichura Theorem (see pp. 1171 in \cite{Knight1989}), in which
we can get a.s. convergence of order statistics of the block lengths, as in
\cite{Knight1989}. Then, in that space we apply the CLT described in Lemma
\ref{clt_general_iid} to obtain the convergence in probability of the bootstrap quantity
$H_{RBB}$ to the CDF of a normal distribution, which implies convergence of the same
things in distribution in the original space. But since this bootstrap limit is
non-random (it does not depend on the data), we get in turn the weakly convergence in
probability. The regenerative block bootstrap is thus first-order correct. In particular,
this justifies the use of the quantiles of the bootstrap distribution (with or without
standardizing) to obtain confidence intervals for $\int f d\pi_{\atom}$.


\begin{remark}

    In the original formulation of the RBB for atomic and positive recurrent chains
    \cite[Theorem 2.1]{Bertail2006}, the estimator used was
    $G_n={\sum_{i=1}^{\numreg}{\fblock[i]}}/{n_{\atom}}$, where
    $n_{\atom}=\sum_{k=1}^{\numreg}{\ell(\block_k)}$. A key element in their proof is
    that $n_{\atom}$ is $a.s.$ equivalent to a multiple of $n$, however, in the
    null-recurrent scenario, this equivalence does not hold due to the lack of first
    moment for $\ell(\block_1)$. Therefore, we need to use the random normalization. On
    the other hand, Remarks \ref{rm:ch4:non_convergence_null_recurrent} and
    \ref{rm:ch4:negligible_blocks} rule out the use of
    ${\sum_{i=1}^{\numreg}{\fblock[i]}}/{n}$ in the null-recurrence case (it converges
    $a.s.$ to 0), and equation \eqref{eq:ch4:lln} suggests
    $\sum_{j=1}^{\numreg}{\fblock[i]}/{\numreg}$ as its natural replacement. It should be
    pointed out that using ${\sum_{i=1}^{\numreg}{\fblock[i]}}/{u(n)}$ (or $S_n(f)/u(n)$)
    is also not useful, because its limit distribution is a constant multiple of a
    Mittag-Leffler distribution, see \cite[Theorem 2.1]{Chen1999}.The random
    normalization seems unavoidable in the $\beta$-null recurrent scenario.

\end{remark}

\section{The regeneration-based bootstrap algorithm}\label{sec:ch4:regeneration_bb}

In this section, we adapt the \textit{Regeneration-based bootstrap} to the $\beta$-null
recurrent Markov chain scenario.

Similarly to Section \ref{sec:ch4:rbb}, consider observations
$\data=\left(X_{0},\ldots,X_{n}\right)$ drawn from a $\beta$-null recurrent Markov chain
$\chain$ that has an accessible atom $\atom$ known beforehand. Suppose that $f$ is a
function such $\pi_{\atom}(f)$ is finite and the second moment of
$f\left(\block_1\right)$ is also finite. Let $\sigma^2$ represent the variance of
$f\left(\block_1\right)$.

The algorithm we present in this section was introduced in
\cite{AthreyaFuh1992,Somnat-1993} for positive recurrent Markov chains with an accessible
known atom. Similarly to the RBB, it consists on dividing the chain into
$\block_1,\ldots,\block_{\numreg}$ regenerative blocks and then resampling blocks to form
the empirical distribution of $\block_1,\ldots,\block_{\numreg}$. The main difference
between the Regeneration-based bootstrap and the RBB is that in the former, the number of
bootstrapped blocks is $\numreg$, hence, non-random conditionally to $\data$, while in
the latter is random.

The full algorithm is as follows:

\begin{enumerate}
    \item Count the number of visits\ $T_n(\atom)$ to the atom $\atom$ up to time $n$,
          and divide the observed sample path\ $\data=(X_{0},\ldots,X_{n})$ into
          $T_n(\atom)+1$ blocks,\ $\block_{0}$, $\block_{1},\ldots,$
          $\block_{T_n(\atom)-1},$ $\block_{T_n(\atom)\;}^{(n)}$, corresponding to the
          pieces of the sample path between consecutive visits to the atom $\atom$.\ Drop
          the first and last (non-regenerative) blocks. Denote by $\numreg$ the number of
          remaining blocks.

    \item  Draw \(\numreg\) bootstrap data blocks $\mathcal{B}_{1,\numreg}^{\ast},..., $
          $\mathcal{B}_{\numreg,\numreg}^{\ast}$ independently from the empirical
          distribution $F_{n}=\numreg^{-1}\sum_{j=1}^{\numreg}\delta_{\block _{j}} $ of
          the blocks $\{\block_{j}\}_{1\leq j\leq \numreg}$ conditioned on $\data$.

    \item From the bootstrap data blocks generated at step 2, reconstruct a trajectory by
          binding the blocks together, getting the reconstructed sample path
          \begin{equation*}
              X^{\ast (n)}=(\mathcal{B}_{1,\numreg}^{\ast},...,\mathcal{B}_{\numreg,\numreg}^{\ast}).
          \end{equation*}
          Compute the statistic $\ G_{n}^{\ast}=G_{n}\left(X^{\ast(n)}\right).$
    \item If $Std_{n}=Std(\block_{1},\ldots,\block_{\numreg})$ is an appropriate
          standardization of the original statistic $G_{n}$, compute
          $Std_{n}^{\ast}=Std(\block_{1,\numreg}^{\ast\text{
          }},\ldots,\block_{\numreg,\numreg}^{\ast})$.
\end{enumerate}

As in the RBB case, the asymptotic result stated below shows the validity of this
bootstrap scheme when used in estimations of integrals with respect to the invariant
measure. In line with the conventions of bootstrap literature, we will write $Z^{\ast}_n
\xrightarrow{d^{\ast}}_{a.s.} Z$ to denote the weak convergence almost surely along the
data of the bootstrap random variables $Z_n^\ast$ towards $Z$. This means that, for every
$x\in\R$, $\P^{\ast}\left(Z_n^{\ast}\leq x\right)\convergence{a.s.}\P\left(Z\leq
x\right)$. For more details, see \cite[pp. 2250]{CavaliereGeorgiev2020}.

\begin{theorem}[Validity of the Regeneration based bootstrap]\label{th:ch3:regeneration_bootstrap} Under the same hypothesis of Theorem \ref{th:ch3:convergence_rbb}, we have
    \begin{equation*}
        \frac{\sqrt {\numreg}}{\hat{\sigma}_{\numreg}}\left(\frac{1}{\numreg}\sum\limits_{j = 1}^{\numreg} {\left( {\fbootsblock - \hat{\mu}_{\numreg}} \right)} \right)\convergence{d^\ast}_{a.s.}  \Normal\left( {0,1} \right).
    \end{equation*}
\end{theorem}

\begin{remark}In its original formulation for the positive recurrent case, the estimator used was
    ${S_n(f)}/{n}$,
    however, by Remark \ref{rm:ch4:non_convergence_null_recurrent}, it can not be used in the null
    recurrent case.
\end{remark}

\section{Simulations}\label{sec:ch4:simulations}

In order to empirically compare the two bootstrap methodologies described in this paper,
we devote this section to simulation examples. The code for all the experiments is
available at \url{https://github.com/carlosds731/boostrap_markov}

As a model for the experiments, we will consider the simple symmetric random walk in
$\Z$, that is
\begin{equation}\label{eq:ch4:ssrw_z}
    X_t=\begin{cases}
        0                             & ,\quad t=0     \\
        \sum\limits_{k = 1}^t {{Y_k}} & ,\quad t\geq 1
    \end{cases}
\end{equation}
with $P\left( {{Y_i} = 1} \right) = P\left( {{Y_i} =  - 1} \right) = {1}/{2}$. In this random walk, the
state $0$ is an atom and the invariant measure is $\pi_0\left(i\right)\equiv 1$ (see pp.1143 in
\cite{Athreya2015}). Consider the function \(f(k)=\frac{1}{k^2}\) if \(k\neq 0\)
and \(f(0)=0\), then
\begin{equation*}
    \int {f\left( x \right)d\pi_0 \left( x \right)}  = 2\sum\limits_{k = 1}^{ + \infty } {\frac{1}{{{k^2}}}}
    = \frac{{{\pi^2}}}{3}.
\end{equation*}

Our parameter of interest will be \(\int {f\left( x \right)d\pi_0 \left( x \right)}\)
which we will estimate with \(G_n={\sum_{i=1}^{\numreg}{\fblock[i]}}/{\numreg}\). The
bootstrap version of this statistic will be denoted by \(G_{RBB,n}^{\ast}\) in the RBB
case and by \(G_{RgB,n}^{\ast}\) in the regeneration based scenario. The standardized
versions of these statistics are defined as follows:
\begin{align*}
    L_n:=\frac{\sqrt{T(n)}}{\hat{\sigma}_n}\left(G_n-\int f \,d\pi_{\atom}\right)                ,  \\
    L^{\ast}_{RBB,n}:=\frac{\sqrt{\bootsnumreg}}{\hat{\sigma}_n}\left(G^{\ast}_{RBB,n}-G_n\right) , \\
    L^{\ast}_{RgB,n}:=\frac{\sqrt{\numreg}}{\hat{\sigma}_n}\left(G^{\ast}_{RgB,n}-G_n\right).
\end{align*}

By Corollary \ref{cor:clt_studentized} and Theorems \ref{th:ch3:convergence_rbb} and
\ref{th:ch3:regeneration_bootstrap} we have that \(L_n\convergence{d}\Normal(0,1)\),
\(L^{\ast}_{RBB,n}\convergence{d^{\ast}}_p\Normal(0,1)\) and
\(L^{\ast}_{RgB,n}\convergence{d^{\ast}}_{a.s.}\Normal(0,1)\).

\subsection{Comparisons with the true distribution}

To see the finite sample performance of both bootstrap methods, and compare its accuracy,
for different values of \(n\) we have simulated a realization of the chain of length
\(n\) and then applied both bootstrap methods \(10^4\) times, obtaining that many samples
of \(L^{\ast}_{RBB,n}\) and \(L^{\ast}_{RgB,n}\). We have then computed the empirical
cumulative distribution function of these statistics and compared with the CDF of
\(L_n\).

The results of these simulations, presented in Figure
\ref{fig:comparison_distribution_functions}, show that as \(n\) increases, the
distributions of both \(L^{\ast}_{RBB,n}\) and \(L^{\ast}_{RgB,n}\) approximate the true
distribution of \(L_n\). Regarding the accuracy, the experiment gives empirical evidence,
that, as in the positive recurrent case, the RBB provides a more accurate approximation
of the true distribution than the regeneration based bootstrap or the asymptotic normal
distribution.

\begin{remark}
    For each \(n\), the true distribution of \(L_n\) is unknown. To obtain a reliable
    approximation of its cumulative distribution function, we simulated \(10^5\) independent
    realizations of \(\chain\) of length \(n\). We then used these samples to compute the
    empirical cumulative distribution function of \(L_n\).
\end{remark}

\begin{figure}[!ht]
    \centering
    \begin{subfigure}[b]{0.47\textwidth}
        \centering
        \includegraphics[width=\textwidth, height=0.2\textheight]{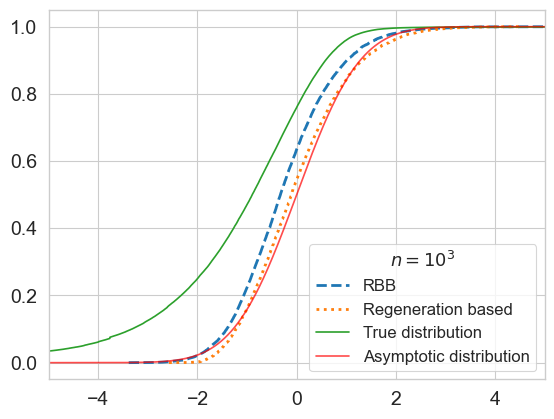}
    \end{subfigure}
    \hfill
    \begin{subfigure}[b]{0.47\textwidth}
        \centering
        \includegraphics[width=\textwidth, height=0.2\textheight]{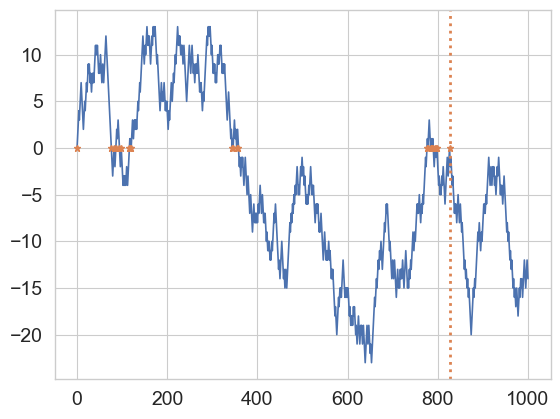}
    \end{subfigure}

    \begin{subfigure}[b]{0.47\textwidth}
        \centering
        \includegraphics[width=\textwidth, height=0.2\textheight]{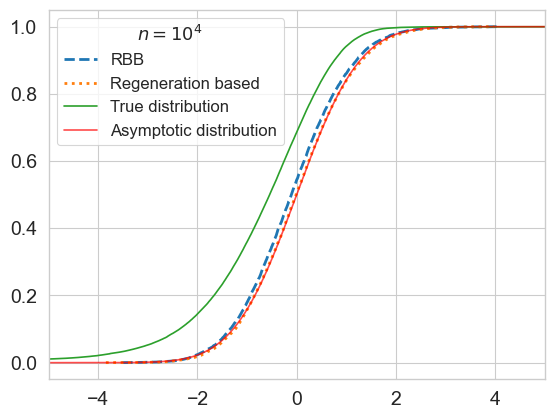}
    \end{subfigure}
    \hfill
    \begin{subfigure}[b]{0.47\textwidth}
        \centering
        \includegraphics[width=\textwidth, height=0.2\textheight]{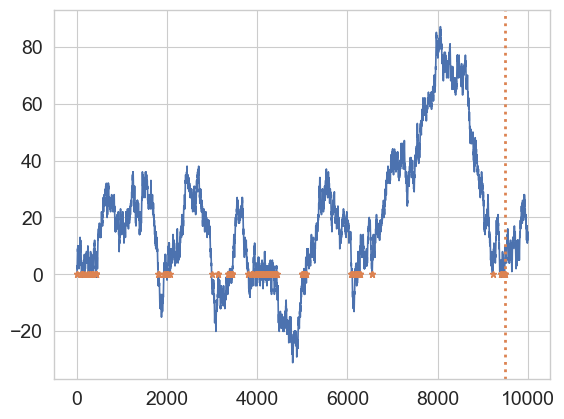}
    \end{subfigure}

    \begin{subfigure}[b]{0.47\textwidth}
        \centering
        \includegraphics[width=\textwidth, height=0.2\textheight]{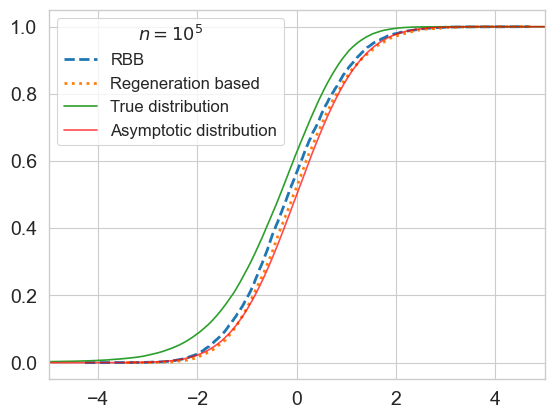}
    \end{subfigure}
    \hfill
    \begin{subfigure}[b]{0.47\textwidth}
        \centering
        \includegraphics[width=\textwidth, height=0.2\textheight]{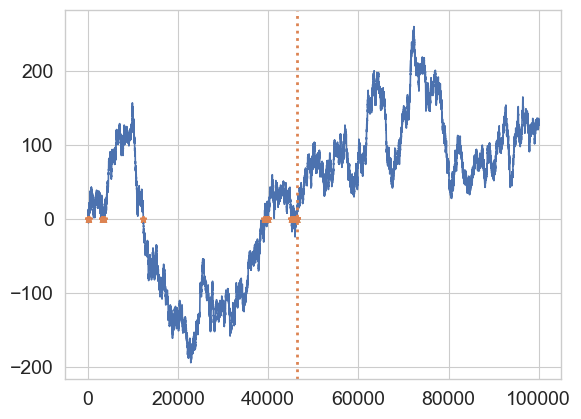}
    \end{subfigure}

    \begin{subfigure}[b]{0.47\textwidth}
        \centering
        \includegraphics[width=\textwidth, height=0.2\textheight]{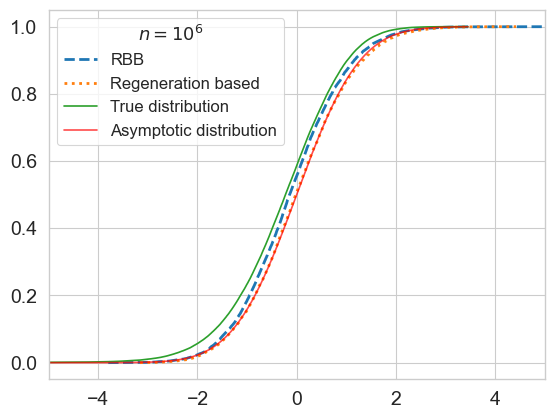}
    \end{subfigure}
    \hfill
    \begin{subfigure}[b]{0.47\textwidth}
        \centering
        \includegraphics[width=\textwidth, height=0.2\textheight]{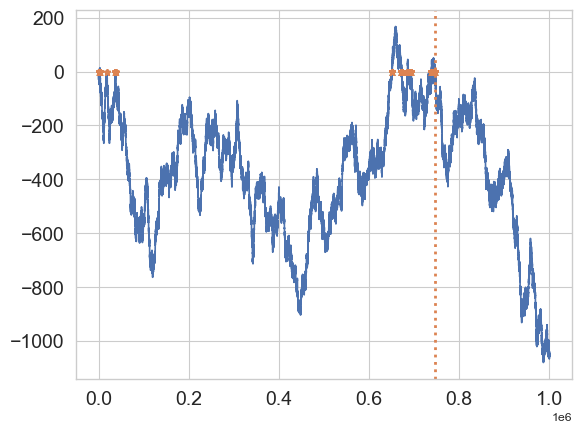}
    \end{subfigure}

    \caption{
        Distribution functions of \(L^{\ast}_{RBB,n}\) and \(L^{\ast}_{RgB,n}\), the
        true distribution of \(L_n\) and the asymptotic distribution (standard normal) for different values of \(n\) (left column) and the
        realization of \(\chain\) from where the samples of \(L^{\ast}_{RBB,n}\) and
        \(L^{\ast}_{RgB,n}\) were obtained (right column), the
        orange stars mark the regeneration times, while the
        orange dotted lines indicate the end of the last complete
        block. }\label{fig:comparison_distribution_functions}
\end{figure}

\subsection{Coverage probability}

The bootstrap methods' first-order correctness established in this paper allows us to use
the quantiles of \(L^{\ast}_{RBB,n}\) and \(L^{\ast}_{RgB,n}\) to construct confidence
intervals for \(\int f \,d\pi_{\atom}\). Let \(q^{\ast}_{RBB}(\alpha)\) and
\(q^{\ast}_{RgB}(\alpha)\) represent the \(\alpha\)-quantiles of \(L^{\ast}_{RBB,n}\) and
\(L^{\ast}_{RgB,n}\) respectively. The bootstrap confidence intervals are then given by:
\begin{align*}
    I^{\ast}_{RBB,n} & =\left[G_n-\frac{\hat{\sigma}_n}{\sqrt{\numreg}}q^{\ast}_{RBB}(1-\alpha/2), G_n-\frac{\hat{\sigma}_n}{\sqrt{\numreg}}q^{\ast}_{RBB}(\alpha/2)\right], \\
    I^{\ast}_{RgB,n} & =\left[G_n-\frac{\hat{\sigma}_n}{\sqrt{\numreg}}q^{\ast}_{RgB}(1-\alpha/2), G_n-\frac{\hat{\sigma}_n}{\sqrt{\numreg}}q^{\ast}_{RgB}(\alpha/2)\right].
\end{align*}

Figure \ref{fig:ch4:coverage_probability} shows the coverage probabilities of
\(I^{\ast}_{RBB,n}\) and \(I^{\ast}_{RgB,n}\) for \(\alpha=0.05\) for different values of
\(n\) while figure \ref{fig:ch4:average_interval_size} shows the average length of these
confidence intervals. For comparisons, we have also included in figures
\ref{fig:ch4:coverage_probability} and \ref{fig:ch4:average_interval_size} the coverage
probabilities and average interval length of the confidence intervals obtained when we
use the normal approximation.

As expected, as \(n\) gets larger, the coverage probability of the confidence intervals
increases, approaching the desired level (\(0.95\)), while the average length decreases.
This experiment reinforces the idea that the RBB provides a better approximation than the
regeneration based bootstrap, as it produces confidence intervals with higher coverage
probability and very similar length. In comparison with the asymptotic distribution, the
RBB generates confidence intervals with narrower lengths and similar coverage
probabilities. This could be explained by noticing that the asymptotic distribution,
being symmetric, does not take into account the possible asymmetry of the underlying
distribution.

\begin{remark}

    To determine the coverage probability for a given \(n\), we simulated \(10^4\)
    independent realizations of \(\chain\) with length \(n\). We then applied both
    bootstrap methods \(10^4\) times, generating the same number of samples for
    \(L^{\ast}_{RBB,n}\) and \(L^{\ast}_{RgB,n}\), which were subsequently used to
    compute their bootstrap quantiles.

\end{remark}

\begin{figure}[!ht]
    \centering
    \begin{subfigure}[b]{0.48\textwidth}
        \centering
        \includegraphics[width=\textwidth]{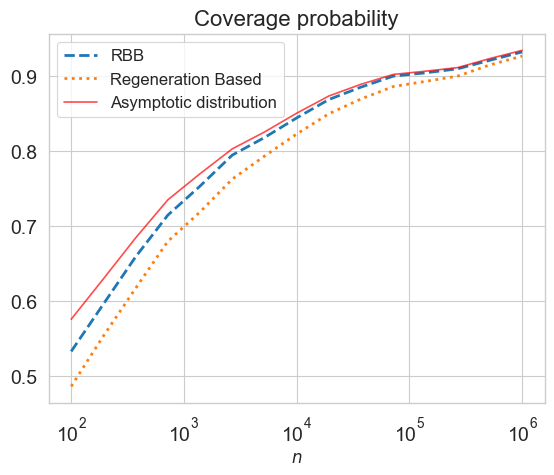}
        \caption{Coverage probabilities}
        \label{fig:ch4:coverage_probability}
    \end{subfigure}
    \hfill
    \begin{subfigure}[b]{0.48\textwidth}
        \centering
        \includegraphics[width=\textwidth]{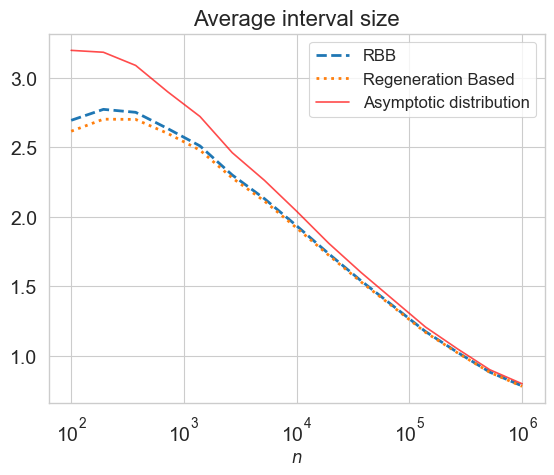}
        \caption{Average interval length}
        \label{fig:ch4:average_interval_size}
    \end{subfigure}
    \caption{
        Coverage probabilities and average interval length of the \(I^{\ast}_{RBB,n}\)
        and \(I^{\ast}_{RgB,n}\) at \(95\%\) confidence level. The x-axis is in logarithmic
        scale.
    }
\end{figure}

\section{Conclusions and perspectives}\label{sec:conclusions}

In this work we have proved the first order validity of the Regenerative block-bootstrap
and Regeneration based bootstrap for atomic $\beta$-null recurrent Markov chains. Up to
our knowledge these are the first bootstrap method whose validity has been established
for these type of non-stationary Markov chains.

In terms of extending the methods to non-atomic chains, it is possible to apply the
Nummelin splitting technique \cite{Nummelin1978,Nummelin1984}, following the approach
described in \cite[Section 3]{Bertail2006}. This construction involves ``extending'' the
chain to make it atomic, then applying the bootstrap to the extended chain. However, to
establish the validity of these bootstrap procedures for non-atomic $\beta$-null
recurrent chains, several additional steps are required. First, we need to derive a
uniform rate of convergence on a small set for the transition kernel estimator (which has
been done at specific points by \cite{Tjostheim-2001} but not uniformly). We also require
new exponential inequalities to obtain rate of convergence. Finally, it must be shown
that the same type of coupling used in \cite{Bertail2006} still holds in this case. This
will be the subject of further investigations.

\section{Proofs}\label{sec:ch4:proofs}

\subsection{Proof of Lemma \ref{clt_general_iid}}

For the proof of Lemma \ref{clt_general_iid} we need the following result, which appears
as part A.3 of Theorem A.1 in \cite{Tjostheim-2001}.

\begin{lemma}\label{cadlagFiniteConvergenceInf}
    Let \(A_n\) and \(B_n\) be a pair of stochastic processes
    which are càdlàg, where \(A_n\) is non-negative and non-decreasing. Let \(B\) denote
    a Brownian motion defined for \(t\in\mathbb{R}\) and let \(A\) denote a strictly increasing
    non-negative process with independent increments, \(A(0)\equiv 0\) and with no fixed
    jumps. Assume that \(B_n\convergence{\cadlaginf} B\) and \(A_n\convergence{\cadlaginf} A\).
    Then, \(A_n^{(-1)}\convergence{\cadlaginf}A^{(-1)}\) and
    \begin{equation*}
        \left( {A_n^{( - 1)}( t ),\frac{{{B_n}\circ A_n^{( - 1)}( t )}}{{\sqrt {A_n^{( - 1)}( t )} }}} \right) \convergence{d} ({A^{(- 1)}}(t),Z)\quad \forall t \in ( {0,1} ],
    \end{equation*}
    where \(Z\) is standard normal random variable independent of \(A^{(-1)}(t)\).
\end{lemma}

To prove Lemma \ref{clt_general_iid}, let \(W_{k}=\sigma^{-1}( X_{k} - \mu)\), then
\(\{W_{k}\}_{k=1}^{\infty}\) is an i.i.d. sequence with \(\E{W_{k}}=0\) and
\(\Var{}{W_{k}}=1\) for all \(k\). Define the following continuous time process for
\(t\geq 0\)
\begin{equation}
    Q_n(t)=\frac{1}{\sqrt{n}}\sum_{k=1}^{\floor{nt}}{W_{k}}.
\end{equation}

By Theorem 23 and Example 24 in \cite{Pollard1984}, \(Q_n\convergence{\cadlaginf}B\) and
given that \(u_n\) is an unbounded increasing sequence, we also have that \(Q_{u_n}\)
converges weakly to \(B\) in \(\cadlaginf\).

The conditions imposed to the process \(N_n\) allow us to apply Lemma
\ref{cadlagFiniteConvergenceInf} with \(A_n=S_n\) and \(B_n=Q_{u_n}\). Taking into
account that \(N_n\) is equivalent to \(S_{n}^{(-1)}\) we obtain that for all \(t>0\)
\begin{equation}\label{intermedianConvergence}
    \frac{Q_{u_n}\left(N_n(t)\right)}{\sqrt{N_n(t)}}\convergence{d}{\Normal\left(0,1\right)}.
\end{equation}
Using that \(N\left(\floor{nt}\right)=u_n N_n(t)\), we get
\begin{equation}\label{composition}
    {Q_{{u_n}}}({N_n}(t)) = \frac{{\sigma^{ - 1}}}{{\sqrt {{u_n}} }} {\sum\limits_{j = 1}^{N\left( {\left\lfloor {nt} \right\rfloor } \right)} {\left( {{X_j} - \mu} \right)} } ,
\end{equation}
and Lemma \ref{clt_general_iid} follows after plugging (\ref{composition}) into (\ref{intermedianConvergence}) and taking \(t=1\).

\subsection{Proof of Corollary \ref{cor: clt_original}}

We assume, at first, that \(\theta\) is bounded, that is, there exists a constant \(K\)
such that \(0<\theta<K\) with probability 1. Without loss of generality, assume the
\(u_n\) are integers. Define the process
\begin{equation*}
    N_n(t)=\begin{cases}
        \frac{t N\left(n\right)}{u_n} & , \text{ if }\frac{N(n)}{u_n}<1 \\
        t\theta                       & , \text{ otherwise }
    \end{cases}.
\end{equation*}

As stated in pp. 147 of \cite{Billingsley1968}, this process converges to the process
\(t\theta\) and trivially satisfies the conditions of Lemma \ref{clt_general_iid} (using
\(S_n(t)=\frac{t}{\theta}\), \(S^{-1}_n(t)=t\theta\)).

The case when \(K\) is unbounded can be treated by following the same argument as in pp.
148 of \cite{Billingsley1968}.

\subsection{Proof of Theorem \ref{cltLemma}}\label{subsec:ch3:proof_clt_chain}

Recall from Section \ref{sec:ch4:beta_null_recurrent} that, by the Strong Markov
Property, the sequence \(\left \{ f(\mathcal{B}_{j}) \right \}_{j=1}^{+\infty}\) is
i.i.d. with mean \(\int f d\pi_{\atom}\) and variance \(\sigma^2\). Consider the
processes $T_n(t)$ and $C_n$ defined in \eqref{eq:ch4:t_n_process}
\begin{equation*}
    T_n(t)=\frac{T\left(\left\lfloor nt \right\rfloor\right)}{u\left( n \right)}\quad,\quad C_n(t)=\frac{1}{v\left( n \right)}\sum_{k=0}^{\left\lfloor nt \right\rfloor}{\ell\left(\block_k\right)}.
\end{equation*}

By Theorem \ref{th:ch4:convergence_process}, we can apply Lemma \ref{clt_general_iid}
with \(X_{i}=f(\mathcal{B}_{i})\), \(\mu=\int f d\pi_{\atom}\), \(N(n)=T(n)\) and
\(u_n=n^\beta L(n)\), which completes the proof of \eqref{clt}. In order to prove
\eqref{clt_non_blocks}, denote by \(W_n\) the left-hand side of \eqref{clt}, then
\begin{equation*}
    \frac{\sqrt{T(n)}}{\sigma}\left(\frac{S_n(f)}{T(n)}-\int f \,d\pi_{\atom}\right)=W_n+\frac{\fblock[0]}{\sqrt{\numreg}}+\frac{\fblock[(n)]}{\sqrt{\numreg}},
\end{equation*}
hence, \eqref{clt_non_blocks} will follow from \eqref{clt} if we show that ${\fblock[0]}/{\sqrt{\numreg}}$ and ${\fblock[(n)]}/{\sqrt{\numreg}}$ converge to 0 at least in
probability.

The random variable $\fblock[0]$ is almost surely bounded and $\numreg$ converges almost
surely to $+\infty$, therefore ${\fblock[0]}/{\sqrt{\numreg}}$ converges to 0 almost
surely. For the other term, first notice that for all $n\in\N$ we have
\begin{align*}
    \frac{|\fblock[(n)]|}{\sqrt{\numreg}}\leq \frac{|f|\left(\block_{(n)}\right)}{\sqrt{\numreg}} \leq \frac{|f|\left(\block_{T(n)+1}\right)}{\sqrt{\numreg}}.
\end{align*}

The random variables \(\{ |f|(\block_j)\}_{j\geq 1}\) are i.i.d. with finite second
moment, therefore, by Lemma 1 in \cite{Athreya-2016}, \(|f|(\block_n)/\sqrt{n}\)
converges to 0 a.s. Since \(\numreg\) converges to \(+\infty\) almost surely, Theorem
6.8.1 in \cite{Gut2013} implies that ${|f|(\block_{T(n)+1})}/{\sqrt{\numreg}}$ converges
to 0 with probability 1, which concludes the proof of Theorem \ref{cltLemma}.

\subsection{Proof of Theorem \ref{th:ch3:convergence_rbb}}

Assume we have observed the chain until time \(n\), i.e., \(\data=X_0,X_1,\ldots,X_n\),
and we have extracted the \(\numreg\) regeneration blocks:
\(\block_1,\ldots,\block_{\numreg}\).

Now we start to sequentially bootstrap data blocks
\(\bootsblock_{1,\numreg},\ldots,\bootsblock_{k,\numreg}\) independently from the
empirical distribution $F_{\numreg}=\numreg^{-1}\sum_{j=1}^{\numreg}\delta_{\block_{j}}$
of the blocks $\{\block_{j}\}_{1\leq j\leq \numreg}$, conditioned on $\data$, until the
length $\ell^{\ast}(k)=\sum_{j=1}^{k}\ell\left(\bootsblock_{1,\numreg}\right)$ of the
bootstrap data series is larger than $n$.

For each \(m\), define
\begin{align}
    T^{\ast}\left(m,\numreg\right)   & =\max\left\{k:\sum_{j=1}^{k}{\ell\left(\fbootsblock[j,\numreg]\right)}\leq m\right\},                                                           \\
    U^{\ast}\left( m,\numreg \right) & =\frac{\sqrt{m}}{\hat{\sigma}_{\numreg}}\left(\frac{1}{m}{\sum\limits_{j=1}^{m}{\left( \fbootsblock[j,\numreg]-\mu_{\numreg} \right)}} \right).
\end{align}

Theorem \ref{th:ch3:convergence_rbb} will be proved if we show that, for all $x\in\R$ it
holds that
\begin{equation}\label{eq:ch3:bootstrap_convergence_in_proba}
    \P^{\ast}\Biggl( U^{\ast}\Bigl(\bootsnumreg,\numreg\Bigr)\leq x \Biggr)\convergence{p}\Phi(x),
\end{equation}
where \(\Phi\) is the cumulative distribution function of a standard normal random
variable and $\P^{\ast}(\bullet)=\P(\bullet\mid \data)$
denotes the conditional probability given $\data$.

Given that we will bootstrap \(\bootsnumreg\) terms, which is a random quantity
conditionally to the data, we will use Lemma \ref{cadlagFiniteConvergenceInf} to prove
\eqref{eq:ch3:bootstrap_convergence_in_proba}. In order to do this we need, conditionally
to the data:
\begin{enumerate}
    \item Find a process \(S^{\ast}_{n,\numreg}(t)\) that is non-negative, non-decreasing
          that converges in \(\cadlaginf\) to a process \(S^{\ast}\) that is
          non-negative, strictly increasing, has independent increments, no fixed jumps
          and \(S^{\ast}\left(0\right)\equiv 0\).
    \item Show that \(T^{\ast}_{n,\numreg}(t)={T^{\ast}\left( \left\lfloor nt
          \right\rfloor \right)}/{\numreg}={T^{\ast}\left( \left\lfloor nt \right\rfloor,
          \numreg \right)}/{\numreg}\) is equivalent in \(\cadlaginf\) to
          \(S_{n,\numreg}^{\ast(-1)}\).
    \item Find a process \(Q_{n,\numreg}^{\ast}(t)\) that converges in \(\cadlaginf\) to
          a Brownian motion when \(n\) goes to \(+\infty\). This process should satisfy,
          for some \(t>0\)
          \begin{equation}\label{eq:ch3:required_equivalance}
              U^{\ast}\Bigl(\bootsnumreg,\numreg\Bigr) = \frac{Q^{\ast}_{n,\numreg}\circ T^{\ast}_{n,\numreg}(t)}{\sqrt{T_{n,\numreg}^{\ast}(t)}}.
          \end{equation}
\end{enumerate}

A natural choice for \(Q_{n,\numreg}^{\ast}\), which satisfies
\eqref{eq:ch3:required_equivalance} for \(t=1\), is
\begin{equation}\label{eq:ch3:expression_q_star}
    Q^{\ast}_{n,\numreg}(t)=\frac{\sqrt{\numreg}}{\hat{\sigma}_{\numreg}}\left( \frac{1}{\numreg}{\sum\limits_{j=1}^{\left\lfloor \numreg t \right\rfloor}{\left( f( {\mathcal{B}_{j,\numreg}^{\ast}} ) - \mu_{\numreg} \right)}}\right).
\end{equation}

Take \(S_{n,\numreg}^{\ast}(t)\) as
\begin{equation}
    S_{n}^{\ast}(t) = \frac{1}{{{v^{\ast}}(\numreg)}}\sum\limits_{i = 1}^{\left\lfloor {\numreg t} \right\rfloor } {\ell\left( \bootsblock_{i,\numreg} \right)},
\end{equation}
where \({v^{\ast}}(\numreg)=\sum_{i=0}^{\numreg}{\ell( {\mathcal{B}_i})}\).

Following the notation of \cite{Knight1989}, let \(Y_i={\ell( {\mathcal{B}_i} )}\) and
let \(Y_{1,n}\geq Y_{2,n}\geq \dots \geq Y_{n,n}\) be the order statistics of the sizes
of the first \(n\) blocks, and take \(Z_{k,n}={Y_{k,n}}/{v(n)}\) where \(v(n)\) is as in
\eqref{eq:ch4:def_v_n}. By Theorem 1 in \cite{Knight1989},
\begin{equation}\label{eq:convergence_distribution}
    {Z^{\left( n \right)}} = \left( {{Z_{1,n}},{Z_{2,n}}, \ldots ,{Z_{n,n}},0,0, \ldots} \right)\convergence{d} \left( {{Z_1},{Z_2}, \ldots} \right) = Z,
\end{equation}
where \({Z_k} = {\left( {{E_1} +  \cdots  + {E_k}} \right)^{ - \frac{1}{\beta }}}\) and
\(E_i\) is a sequence of i.i.d. of exponential random variables with mean 1.
By Skorokhod-Dudley-Wichura Theorem (see pp. 1171 in \cite{Knight1989} and pp. 476 in
\cite{Berkes2010}) we can choose a probability space such that, without changing the
distribution of the left-hand side of \eqref{eq:convergence_distribution},
\begin{equation}\label{eq:ch3:almost_sure_representation}
    {Z^{\left( n \right)}} \convergence{a.s.} Z.
\end{equation}

The following Lemma shows that in that space, conditionally to the data, the process
\(S^{\ast}_{n,\numreg}\) converges in \(\cadlaginf\).

\begin{lemma}\label{lemma:ch3:convergence_s_n_star} Suppose that \eqref{eq:ch3:almost_sure_representation} holds. Let \(\lambda^{\ast}_{j}(t)\) be independent Poisson processes with
    parameter \(1\), and \(K\) a positive constant. Define
    \begin{equation*}
        R^{\ast}(t)=\sum_{j=1}^{+\infty}{{Z_j}\left(\lambda^{\ast}_{j}(t)-t\right)} \quad\textnormal{and}\quad S^{\ast}(t)=KR^{\ast}(t)+t.
    \end{equation*}
    Then, \({\numreg}/{u(n)}\) converges almost surely to a positive random variable
    and
    \begin{equation}\label{eq:ch3:convergence_s_star_process}
        S^{\ast}_{n,\numreg}\convergence{\cadlaginf}S^{\ast}\quad\textnormal{and}\quad S^{\ast(-1)}_{n,\numreg}\convergence{\cadlaginf}S^{\ast(-1)}
    \end{equation}
    almost surely along the data. Moreover, the process \(S^{\ast}\)
    is non-negative, strictly increasing, continuous, with independent increments and
    \(S^{\ast}\left(0\right)\equiv 0\).
\end{lemma}
\begin{proof}
    When \eqref{eq:ch3:almost_sure_representation} holds, by Theorem 1 and Remark 1.3
    in \cite{LePage1981},
    \begin{equation*}
        \frac{1}{v(n)}\sum_{j=1}^{n}{\ell\left(\block_j\right)}\convergence{a.s.}\sum_{j=1}^{+\infty}{Z_j}.
    \end{equation*}

    The length of the first block, $\ell\left(\block_0\right)$, is finite with
    probability $1$ and does not depend on $n$, hence
    ${\ell\left(\block_0\right)}/{v(n)}$ converges almost surely to 0. This implies that
    \begin{equation}\label{eq:ch3:a_s_convergence_sum}
        \frac{1}{v(n)}\sum_{j=0}^{n}{\ell\left(\block_j\right)}\convergence{a.s.}\sum_{j=1}^{+\infty}{Z_j}.
    \end{equation}

    In \eqref{eq:ch4:def_v_n}, we defined \(v(z)\) as the inverse of \(u(z)=z^\beta
    L(z)\), then, by Proposition 1.5.15 in \cite{Bingham1987}, \(v(z)\sim z^{1/\beta}
    L_1(z)\) where \(L_1\) is a slowly varying function, hence, \(v(n)/v(\lfloor nt
    \rfloor)\to t^{-1/\beta}\), and we have that
    \begin{equation}\label{eq:ch3:a_s_convergence_sum_t}
        \frac{1}{v(n)}\sum_{j=0}^{\left\lfloor nt \right\rfloor}{\ell\left(\block_j\right)}\convergence{a.s.}t^{\frac{1}{\beta}}\sum_{j=1}^{+\infty}{Z_j}\quad\forall t>0.
    \end{equation}

    For each \(t>0\), let
    \begin{equation*}
        S_n(t)= \frac{1}{v(n)}\sum_{j=0}^{\left\lfloor nt \right\rfloor}{\ell\left(\block_j\right)}\; ,\;S_n^{(-1)}(t) = \inf\{ x>0: S_n( x )>t\}\; ,\;S(t)=t^{{1}/{\beta}}\sum_{j=1}^{+\infty}{Z_j},
    \end{equation*}
    and define the three processes as \(0\) on \(t=0\). By \eqref{eq:ch3:a_s_convergence_sum_t} and the
    Continuous Mapping Theorem, \(S_n^{(-1)}\convergence{a.s.}S^{-1}\).

    Similar to what is described on page 1141 in \cite{Athreya2015}, suppose that $y$ is
    such that $y<S_n^{(-1)}(1)$. Then, since $S_n(y)<1$, it follows that
    $\sum_{j=0}^{\lfloor ny \rfloor} \ell(\block_j) < v(n)$. Consequently, we have
    $T(\lfloor v(n) \rfloor) \geq \lfloor ny \rfloor > ny - 1$. This in turn implies that
    ${T(\lfloor v(n) \rfloor)}/{n} \geq y - {1}/{n} \geq S_n^{(-1)}(1) - {1}/{n}$ for all
    \(n\). Similarly, but taking $y>S_n^{(-1)}$, we show that ${T(\lfloor v(n)
    \rfloor)}/{n} \leq S_n^{(-1)}(1) + {1}/{n}$ for all $n$. Then,
    \begin{equation}\label{eq:ch3:t_n_bound_a_s}
        S_{u(n)}^{{(-1)}}\left(1\right)-\frac{1}{u\left( n \right)}\leq \frac{T\left( \left\lfloor v\left( u\left( n \right) \right) \right\rfloor \right)}{u\left( n \right)}\leq S_{u(n)}^{(-1)}\left(1\right)+\frac{1}{u\left( n \right)}.
    \end{equation}

    The first part of the lemma now follows from \eqref{eq:ch3:t_n_bound_a_s}, the
    convergence of \(S^{(-1)}_{u(n)}\left(1\right)\) to \(S^{-1}\left(1\right)\) and the
    fact that \(u\left(v(n)\right)=n\) for \(n\) big enough.

    To show \eqref{eq:ch3:convergence_s_star_process}, consider the following process,
    which was studied in \cite{Berkes2010},
    \begin{equation*}
        Z^{\ast}_{m,n}(t)=\frac{1}{v\left( n \right)}\sum\limits_{j=1}^{\left\lfloor mt \right\rfloor}\left( \ell\left( \bootsblock_{j,n} \right) - \frac{1}{n}{\sum_{i=1}^{n}{\ell\left( \block_i \right)}}  \right).
    \end{equation*}

    By Corollary 1.2 in \cite{Berkes2010} (and its proof\footnote{In \cite{Berkes2010},
    they standardize by \(T_n=\max\limits_{1\leq k\leq n} l\left(\block_k\right)\) but
    from the proof is clear that the result remains valid if we standardize by \(v(n)\)
    ($b_n$ in their notation).}), we see that when
    \eqref{eq:ch3:almost_sure_representation} holds, for any \(m_n\) such that
    \({m_n}/{n}\to c\), conditionally to the data, the process $Z^{\ast}_{m_n,n}$
    converges weakly in $\cadlagunit$ to $R^{\ast}\left(ct\right)$. Let \(C>1\), on
    $[0,C]$ define the process
    \begin{equation*}
        W^{\ast}_{n}(t)=\frac{1}{v\left( n \right)}\sum\limits_{j=1}^{\left\lfloor nt \right\rfloor}\left( \ell\left( \bootsblock_{j,n} \right) - \frac{1}{n}{\sum_{i=1}^{n}{\ell\left( \block_i \right)}}\right).
    \end{equation*}
    Notice that $W^{\ast}_{n}(t)=Z^{\ast}_{nC,n}\left({t}/{C}\right)$,
    hence, \(W^{\ast}_{n}\convergence{\cadlag_{[0,C]}}R^{\ast}\) as \(n\to +\infty\).
    Because this convergence holds for arbitrary \(C>0\), by Lemma 1.3.ii in \cite{Kasahara1984}
    we have that $W^{\ast}_{n}\convergence{\cadlaginf}R^{\ast}$, and therefore,
    $W^{\ast}_{\numreg}\convergence{\cadlaginf}R^{\ast}$.

    The process \(S^{\ast}_{n,\numreg}\) can be written as
    \begin{equation}\label{eq: s_n_star_process}
        S_{n,\numreg}^{\ast}(t)=\frac{v(\numreg)}{v^{\ast}(\numreg)}W_{\numreg,\numreg}^{\ast}(t)+\frac{\left\lfloor \numreg t \right\rfloor}{\numreg}.
    \end{equation}

    Notice that
    \begin{equation*}
        \frac{v(\numreg)}{v^{\ast}(\numreg)}=\left( \frac{1}{v\left(\numreg \right)}\sum_{j=0}^{\numreg}{\ell\left(\block_j\right)} \right)^{-1},
    \end{equation*}
    then, conditionally to the data, it converges to a positive constant \(K\) by equation
    \eqref{eq:ch3:a_s_convergence_sum}. Equation
    \eqref{eq:ch3:convergence_s_star_process} now follows from the convergence of
    $W^{\ast}_{\numreg,\numreg}$ and \eqref{eq: s_n_star_process}. The continuity of
    \(S^{\ast}\) was shown in pp. 466 of \cite{Berkes2010}, and the rest of the
    properties are directly deduced from the form of \(R^{\ast}\).
\end{proof}

The next Lemma handles the equivalence of \(T^{\ast}_{n,\numreg}\) and
\(S_{n,\numreg}^{\ast(-1)}\) in \(\cadlaginf\).

\begin{lemma}\label{lemma:t_t_n_equivalence_bootstrap} Under the same hypothesis of
    Lemma \ref{lemma:ch3:convergence_s_n_star}, the processes \(T^{\ast}_{n,\numreg}\) and
    \(S_{n,\numreg}^{\ast(-1)}\) are equivalent in \(\cadlaginf\).
\end{lemma}
\begin{proof}
    The proof of this result follows the proof of Theorem 3.2 on \cite{Tjostheim-2001} with slight modifications.

    We need to show that, for any \(\epsilon>0\) given,
    \begin{equation}\label{eq:ch3:condition_bootstrap_convergence}
        \P\left( {\mathop {\sup }\limits_{0 < t \leq K} \left| {T_{n,\numreg}^{\ast}(t) - S_{n,\numreg}^{*\left( { - 1} \right)}(t)} \right| >\epsilon } \right) \to 0\quad \forall K>0.
    \end{equation}

    To prove this, we will show that
    \begin{align}
        \P\left( {\mathop {\sup }\limits_{0 < t \leq K} \left| {T_{{v^{\ast}}(\numreg),\numreg}^{\ast}(t) - S_{n,\numreg}^{*\left( { - 1} \right)}(t)} \right| >\epsilon } \right) \to 0\quad \forall K>0,\label{eq:ch3:equivalence_t_v_t_n} \\
        \P\left( {\mathop {\sup }\limits_{0 < t < K} \left| {T_{v^{\ast}(\numreg),\numreg}^{\ast}(t) - T_{n,\numreg}^{\ast}(t)} \right| > \varepsilon } \right) \to 0\quad \forall K > 0.\label{eq:condition_bootstrap_convergence_t_n_star}
    \end{align}
    from where \eqref{eq:ch3:condition_bootstrap_convergence} will follow by triangular
    inequality.

    Let \(\eta>0\)
    \begin{align}
        \left\{ {S_{n,\numreg}^{*\left( { - 1} \right)}(t) < \eta } \right\} & \subseteq \left\{ {S_{n,\numreg}^{\ast}\left( \eta  \right) > t} \right\}\nonumber                                                                                                                                                              \\
                                                                             & = \left\{ {\frac{1}{{{v^{\ast}}\left( {\numreg} \right)}}\sum\limits_{i = 1}^{\left\lfloor {{\numreg}\eta} \right\rfloor } {\ell\left( {\mathcal{B}_{i,{\numreg}}^{\ast}} \right)}  > t} \right\}\nonumber                                      \\
                                                                             & = \left\{ {\sum\limits_{i = 1}^{\left\lfloor {{\numreg}\eta } \right\rfloor } {\ell\left( {\mathcal{B}_{i,{\numreg}}^{\ast}} \right)}  > t{v^{\ast}}\left( {\numreg} \right)} \right\}\nonumber                                                 \\
                                                                             & = \left\{ {\frac{{{T^{\ast}}\left( {\left\lfloor {{v^{\ast}}\left( {\numreg} \right)}t \right\rfloor },\numreg \right)}}{{\numreg}} < \frac{{\left\lfloor {{\numreg}\eta } \right\rfloor }}{{\numreg}}} \right\}.\label{eq:s_u_star_and_t_star}
    \end{align}

    Because \(T\left({{v^{\ast}}\left( n \right)}\right) = n\), we can write,
    \[T_{{v^{\ast}}\left( {\numreg} \right),\numreg}^{\ast}(t) = \frac{{{T^{\ast}}\left( {\left\lfloor {{v^{\ast}}\left( {\numreg} \right)t} \right\rfloor },\numreg \right)}}{{{u^{\ast}}\left( {{v^{\ast}}\left( {\numreg} \right)} \right)}} = \frac{{{T^{\ast}}\left( {\left\lfloor {{v^{\ast}}\left( {\numreg} \right)t} \right\rfloor },\numreg \right)}}{{\numreg}},\]
    therefore, equation (\ref{eq:s_u_star_and_t_star}) becomes
    \begin{equation}\label{eq:lower-bound}
        \left\{ {S_{n,\numreg}^{*\left( { - 1} \right)}(t) < \eta } \right\} \subseteq \left\{ {T_{v^{\ast}(\numreg),\numreg}^{\ast}(t) < \frac{{\left\lfloor {{\numreg}\eta } \right\rfloor }}{{\numreg}}} \right\}.
    \end{equation}

    Similarly, we obtain that
    \begin{equation}\label{eq:upper-bound}
        \left\{ {S_{n,\numreg}^{*\left( { - 1} \right)}(t) > \eta } \right\} \subseteq
        \left\{ {T_{v^{\ast}(\numreg),\numreg}^{\ast}(t) \geq
                \frac{{\left\lfloor {{\numreg}\eta } \right\rfloor }}{{\numreg}}} \right\}.
    \end{equation}

    Let \(\epsilon_1\in(0,1)\) be fixed and take \(\eta_1<\eta_2\), then, by
    (\ref{eq:upper-bound}) and (\ref{eq:lower-bound}),
    \begin{align*}
        \left\{ {{\eta _1} \leq S_{n,\numreg}^{*\left( { - 1} \right)}(t) < {\eta _2}} \right\} & \subseteq \left\{ {{\eta _1}\left( {1 - {\epsilon _1}} \right) < S_{n,\numreg}^{*\left( { - 1} \right)}(t) < {\eta_2}} \right\}                                                                                                                          \\
                                                                                                & \subseteq \left\{ {\frac{{\left\lfloor {{\numreg}{\eta _1}\left( {1 - {\epsilon _1}} \right)} \right\rfloor }}{{\numreg}} \leq T_{v^{\ast}(\numreg),\numreg}^{\ast}(t) < \frac{{\left\lfloor {{\numreg}{\eta _2}} \right\rfloor }}{{\numreg}}} \right\}.
    \end{align*}

    This means, that, if \(S_{n,\numreg}^{*\left( { - 1} \right)}(t) \in \left[ {{\eta
    _1},{\eta _2}} \right)\), then
    \[\frac{{\left\lfloor {{\numreg}{\eta _1}\left( {1 - {\varepsilon _1}} \right)} \right\rfloor }}{{\numreg}}-\eta_2 < T_{v^{\ast}(\numreg),\numreg}^{\ast}(t) - S_{n,\numreg}^{*\left( { - 1} \right)}(t) < \frac{{\left\lfloor {{\numreg}{\eta _2}} \right\rfloor }}{{\numreg}}-{\eta_1},\]
    which implies that, if $S_{n,\numreg}^{*( - 1)}(t) \in \left[ {{\eta _1},{\eta _2}}
    \right)$, then
    \begin{equation}\label{eq:modules-bound}
        \left| {T_{v^{\ast}(\numreg),\numreg}^{\ast}(t) - S_{n,\numreg}^{*\left( { - 1} \right)}(t)} \right| \leq {\eta _2} - {\eta _1} + {\varepsilon _1}{\eta _1} + \frac{1}{{\numreg}}.
    \end{equation}

    Let \(\varepsilon>0\) be fixed. For any \(s\) we have
    \begin{align*}
        \P\left( {\mathop {\sup }\limits_{t \leq K} \left| {{\xi^{\ast}_{n,\numreg}}(t)} \right| > \varepsilon } \right) & \leq \P\left( {\mathop {\sup }\limits_{t \leq K} \left| {{\xi^{\ast}_{n,\numreg}}(t)} \right| > \varepsilon ,\mathop {\sup }\limits_{t \leq K} S_{n,\numreg}^{*\left( { - 1} \right)}(t) < s} \right) \\
                                                                                                                         & + \P\left( {\mathop {\sup }\limits_{t \leq K} S_{n,\numreg}^{*\left( { - 1} \right)}(t) \geq s} \right),
    \end{align*}
    where \({\xi^{\ast}_{n,\numreg}}(t) = T_{v^{\ast}(\numreg),\numreg}^{\ast}(t) - S_{n,\numreg}^{*\left( { - 1} \right)}(t)\).

    By \eqref{eq:ch3:convergence_s_star_process},
    \[\mathop {\lim }\limits_{s \uparrow \infty } \mathop {\lim }\limits_{n \to \infty } \P\left( {\mathop {\sup }\limits_{t \leq K} S_{n,\numreg}^{*\left( { - 1} \right)}(t) \geq s} \right) = 0.\]

    Therefore, for any \(\delta>0\) we can choose \(s_0\) such that, for \(n\) big
    enough,
    \[\P\left( {\mathop
                    {\sup }\limits_{t \leq K} S_{n,\numreg}^{*\left( { - 1} \right)}(t) \geq s_0}
        \right)<\delta.\]

    By (\ref{eq:modules-bound}), \({\mathop {\sup }\limits_{t \leq K}
    S_{n,\numreg}^{*\left( { - 1} \right)}(t) < s_0}\) implies that
    \[\left| {{\xi^{\ast}_{n,\numreg}}(t)} \right| \leq {\eta _2} - {\eta _1} + {\varepsilon _1}{\eta _1} + \frac{1}{{\numreg}}\quad \forall t \in \left[ {0,K} \right]\;\;,\;\;\forall {\varepsilon _1} \in \left( {0,1} \right).\]

    Choose \(\eta_0,\ldots,\eta_L, N_1, \varepsilon_1\) with
    \(\eta_0=0<\eta_1<\ldots<\eta_{L-1}<\eta_L=s_0\) such that
    \(\eta_i-\eta_{i+1}<{\varepsilon}/{3}\) for all $i$. Let
    \(\varepsilon_1<{\varepsilon}/{s_0}\) and choose \(N_1\) such that
    \({1}/{T\left(N_1\right)}<{\varepsilon}/{3}\).

    Notice that for all \(t\in [0,K]\) there is only one \(i_{n,t}\) such that
    \(S_{n,\numreg}^{*\left( { - 1} \right)}(t)\) belongs to \(\left[ {{\eta
    _{i_{n,t}}},{\eta _{i_{n,t} + 1}}} \right)\), then, by (\ref{eq:modules-bound})
    \[\left| {{\xi^{\ast}_{n,\numreg}}(t)} \right| \leq {\eta_{i_{n,t}}} - {\eta_{i_{n,t} + 1}} + {\varepsilon _1}{\eta _1} + \frac{1}{{\numreg}} \leq \varepsilon \quad \forall t\in[0,K]\;,\;\forall n>N_1,\]
    whenever \(S_{n,\numreg}^{*\left( { - 1} \right)}(t) < s_0\). This implies that
    \[\P\left( {\mathop {\sup }\limits_{t \leq K} \left| {{\xi^{\ast}_{n,\numreg}}(t)} \right| > \varepsilon ,\mathop {\sup }\limits_{t \leq K} S_{n,\numreg}^{*\left( { - 1} \right)}(t) < {s_0}} \right) = 0\quad \forall n \geq {N_1}.\]

    Hence,
    \begin{equation}\label{eq:prob_sup}
        \P\left( {\mathop {\sup }\limits_{t \leq K} \left| {{\xi^{\ast}_{n,\numreg}}(t)} \right| > \varepsilon } \right) < \delta \quad \forall n > {N_1},
    \end{equation}
    which implies (\ref{eq:ch3:equivalence_t_v_t_n}).

    Now we turn to the proof of \eqref{eq:condition_bootstrap_convergence_t_n_star}.

    According to the definition of \(v^{\ast}\), \({v^{\ast}}( {\numreg}) = \sum_{i =
    0}^{\numreg} {\ell\left( {{\mathcal{B}_i}} \right)} \leq n\), therefore,
    \begin{align*}
        T_{v^{\ast}(\numreg),\numreg}^{\ast}(t) & = \frac{{{T^{\ast}}\left( {\left\lfloor {{v^{\ast}}\left( {\numreg} \right)t} \right\rfloor },\numreg \right)}}{{\numreg}}\leq \frac{{{T^{\ast}}\left( {\left\lfloor {nt} \right\rfloor },\numreg \right)}}{{\numreg}} \\
                                                & \leq T_{n,\numreg}^{\ast}(t)\quad \forall n,t.
    \end{align*}

    Notice that ${v^{\ast}}\left( {{\numreg} + 1} \right) = \sum_{i = 0}^{\numreg + 1}
    {\ell( {{\mathcal{B}_i}})} > n,$ therefore,
    \begin{equation*}
        T_{n,\numreg}^{\ast}(t)\leq T^{\ast}_{v^{\ast}\left(\numreg+1\right),\numreg}(t)\frac{\numreg+1}{\numreg}\quad \forall n,\;t.
    \end{equation*}

    Hence,
    \[T_{v^{\ast}(\numreg),\numreg}^{\ast}(t) \leq T_n^{\ast}(t) \leq T_{{v^{\ast}}(\numreg+1)}^{\ast}(t)\frac{\numreg+1}{{\numreg}}\quad \forall n,t.\]

    Equation \eqref{eq:condition_bootstrap_convergence_t_n_star} now follows from the
    convergence of both \(T_{v^{\ast}(\numreg),\numreg}^{\ast}\) and
    \(T_{{v^{\ast}}(\numreg+1),\numreg}^{\ast}\) to $S^{\ast(-1)}$ and the fact that
    \({(\numreg+1)}/{\numreg}\) converges almost surely to 1.
\end{proof}

By \eqref{eq:ch3:expression_q_star}, Lemmas \ref{cadlagFiniteConvergenceInf},
\ref{lemma:ch3:convergence_s_n_star} and \ref{lemma:t_t_n_equivalence_bootstrap} we have
that, in a space where \eqref{eq:ch3:almost_sure_representation} holds, the convergence
in \eqref{eq:ch3:bootstrap_convergence_in_proba} holds almost surely. Therefore, in the
original space we have the weakly-weakly convergence\footnote{The weakly-weakly
convergence, introduced in \cite{CavaliereGeorgiev2020} is the translation of the concept
of weak convergence of random measures to the probabilistic setting, that is, for random
variables \((Z,X)\) and \((Z_n,X_n)\) defined on possibly different probability spaces,
the weakly-weakly convergence of \(Z_n|X_n\) to \(Z|X\) is defined by the fact \(\E
[g(Z_n)|X_n]\convergence{d}\E [g(Z)|X]\) for all bounded and continuous functions \(g\).
For a detailed description of this concept as well as other examples of its application
in the bootstrap setting, please refer to pp. 2550 and Appendix A in
\cite{CavaliereGeorgiev2020}.}
\begin{equation}\label{eq:ch3:weakly_weakly_convergence}
    \forall x\in\R\quad\P^{\ast}\Biggl( U^{\ast}\Bigl(\bootsnumreg,\numreg\Bigr)\leq x \Biggr)\convergence{d}\Phi(x).
\end{equation}

However, given that the right-hand side of \eqref{eq:ch3:weakly_weakly_convergence} is a
constant for each \(x\), the convergence in \eqref{eq:ch3:weakly_weakly_convergence} can
be improved to convergence in probability, which completes the proof.

\subsection{Proof of Theorem \ref{th:ch3:regeneration_bootstrap}}

This proof follows the line of the proof of Theorem 2.1 in \cite{Bickel1981}. As in that
paper, let \(\Gamma_2\) be the set of distribution functions \(G\) satisfying
\(\int{x^2}dG\left(x\right)<\infty\) and define the following notion of convergence in
\(\Gamma_2\)
\begin{equation}\label{eq: ch2: mallows_convergence}
    G_n\Rightarrow G \quad\text{iff}\quad G_n\to G\;\text{weakly}\;\;\text{and}\;\;\int{x^2dG_n\left( x \right)}\to\int{x^2dG\left( x \right)}.
\end{equation}

Denote by \(d_2\) a Mallows metric that metricizes the \(\Rightarrow\) convergence in
\(\Gamma_2\) (see details in Section 8 of \cite{Bickel1981})

If \(Y_1,\ldots,Y_n\) are i.i.d. random variables with common distribution \(G\), denote
by \(G^{(m)}\) the distribution of
\[m^{-\frac{1}{2}}\sum_{j=1}^{m}{\left( Y_j-\E Y_j \right)}.\]

By pp. 1198 in \cite{Bickel1981}, if \(G,H\in\Gamma_2\) then \(G^{(m)}\) and \(H^{(m)}\)
are also in \(\Gamma_2\) and
\begin{equation}\label{eq: ch2: mallows_inequality}
    d_2\left(G^{(m)},H^{(m)}\right)\leq d_2\left(G,H\right).
\end{equation}

Let \(F\) be the distribution of \(\fblock[1]\) and denote by \(F_n\) the empirical
distribution function of \(\fblock[1],\ldots, \fblock[n]\). By (2.1) in \cite{Bickel1981}
and the fact that \(\numreg\to +\infty\) a.s., \(F_{\numreg}\Rightarrow F\) along almost
almost all sample paths, hence, conditionally to the data
\begin{equation}\label{eq: ch2: regeneration_bb_bound_1}
    d_2\left( F_{\numreg},F \right)\to 0.
\end{equation}

Denote by \(N_\sigma\) a standard distribution with mean \(0\) and variance \(\sigma^2\).
By Proposition \ref{cltLemma},
\begin{equation}\label{eq: ch2: regeneration_bb_bound_2}
    d_2\left( F^{(\numreg)},N_\sigma \right)\to 0.
\end{equation}

Conditionally to the data, the distribution of
\[\sqrt{\numreg}\left( \frac{\sum_{j=1}^{\numreg}{\left( f\left( {\mathcal{B}_{j,\numreg}^*} \right) - \frac{1}{{\numreg}}\sum\limits_{i = 1}^{\numreg} {f\left( {\mathcal{B}{_i}} \right)} \right)}}{\numreg} \right)\]
is \(F_{\numreg}^{(\numreg)}\), then, conditionally to the data,
\begin{equation*}
    d_2\left( F_{\numreg}^{(\numreg)},N_\sigma \right)\le d_2\left( F_{\numreg}^{(\numreg)},F^{(\numreg)} \right)+d_2\left( F^{(\numreg)},N_\sigma \right)
\end{equation*}
which goes to 0 by \eqref{eq: ch2: regeneration_bb_bound_1} and
\eqref{eq: ch2: regeneration_bb_bound_2}. The theorem now follows
by \eqref{eq: ch2: mallows_convergence}, \eqref{eq:ch4:varianceEstimator}
and Slutsky's theorem.


\bibliographystyle{imsart-number} 
\bibliography{bibliography}       

\begin{thebibliography}{68}

\bibitem{Alexander2011}
\begin{barticle}[author]
\bauthor{\bsnm{Alexander},~\bfnm{Kenneth}\binits{K.}}
(\byear{2011}).
\btitle{Excursions and Local Limit Theorems for Bessel-like Random Walks}.
\bjournal{Electronic Journal of Probability}
\bvolume{16}.
\bdoi{10.1214/ejp.v16-848}
\end{barticle}
\endbibitem

\bibitem{Anscombe1952}
\begin{barticle}[author]
\bauthor{\bsnm{Anscombe},~\bfnm{F.~J.}\binits{F.~J.}}
(\byear{1952}).
\btitle{Large-sample theory of sequential estimation}.
\bjournal{Mathematical Proceedings of the Cambridge Philosophical Society}
\bvolume{48}.
\bdoi{10.1017/s0305004100076386}
\end{barticle}
\endbibitem

\bibitem{AthreyaFuh1992}
\begin{barticle}[author]
\bauthor{\bsnm{Athreya},~\bfnm{K.~B.}\binits{K.~B.}} \AND
  \bauthor{\bsnm{Fuh},~\bfnm{C.~D.}\binits{C.~D.}}
(\byear{1992}).
\btitle{Bootstrapping Markov chains: countable case}.
\bjournal{Journal of Statistical Planning and Inference}
\bvolume{33}
\bpages{311-331}.
\bdoi{https://doi.org/10.1016/0378-3758(92)90002-A}
\end{barticle}
\endbibitem

\bibitem{Athreya1987}
\begin{barticle}[author]
\bauthor{\bsnm{Athreya},~\bfnm{K.~B.}\binits{K.~B.}}
(\byear{1987}).
\btitle{{Bootstrap of the Mean in the Infinite Variance Case}}.
\bjournal{The Annals of Statistics}
\bvolume{15}
\bpages{724 -- 731}.
\bdoi{10.1214/aos/1176350371}
\end{barticle}
\endbibitem

\bibitem{Athreya2015}
\begin{barticle}[author]
\bauthor{\bsnm{Athreya},~\bfnm{Krishna~B.}\binits{K.~B.}} \AND
  \bauthor{\bsnm{Roy},~\bfnm{Vivekananda}\binits{V.}}
(\byear{2015}).
\btitle{{Estimation of integrals with respect to infinite measures using
  regenerative sequences}}.
\bjournal{Journal of Applied Probability}
\bvolume{52}
\bpages{1133 -- 1145}.
\bdoi{10.1239/jap/1450802757}
\end{barticle}
\endbibitem

\bibitem{Athreya-2016}
\begin{barticle}[author]
\bauthor{\bsnm{Athreya},~\bfnm{Krishna~B.}\binits{K.~B.}} \AND
  \bauthor{\bsnm{Roy},~\bfnm{Vivekananda}\binits{V.}}
(\byear{2016}).
\btitle{General Glivenko-Cantelli Theorems}.
\bjournal{Stat}
\bvolume{5}.
\bdoi{10.1002/sta4.128}
\end{barticle}
\endbibitem

\bibitem{Berkes2010}
\begin{barticle}[author]
\bauthor{\bsnm{Berkes},~\bfnm{Istv{\'a}n}\binits{I.}},
  \bauthor{\bsnm{Horv{\'a}th},~\bfnm{Lajos}\binits{L.}} \AND
  \bauthor{\bsnm{Schauer},~\bfnm{Johannes}\binits{J.}}
(\byear{2010}).
\btitle{Non-central limit theorems for random selections}.
\bjournal{Probability theory and related fields}
\bvolume{147}
\bpages{449--479}.
\end{barticle}
\endbibitem

\bibitem{ber94}
\begin{barticle}[author]
\bauthor{\bsnm{Bertail},~\bfnm{Patrice}\binits{P.}}
(\byear{1994}).
\btitle{{Un test Bootstrap dans un modéle AR(1)}}.
\bjournal{Annals of Economics and Statistics}
\bvolume{36}
\bpages{57-79}.
\end{barticle}
\endbibitem

\bibitem{BertailClemencon2004}
\begin{barticle}[author]
\bauthor{\bsnm{Bertail},~\bfnm{Patrice}\binits{P.}} \AND
  \bauthor{\bsnm{Cl{\'e}men{\c c}on},~\bfnm{St{\'e}phan}\binits{S.}}
(\byear{2004}).
\btitle{Edgeworth expansions of suitably normalized sample mean statistics for
  atomic Markov chains}.
\bjournal{Probability Theory and Related Fields}
\bvolume{130}
\bpages{388--414}.
\bdoi{10.1007/s00440-004-0360-0}
\end{barticle}
\endbibitem

\bibitem{BertailClemencon2007}
\begin{barticle}[author]
\bauthor{\bsnm{Bertail},~\bfnm{Patrice}\binits{P.}} \AND
  \bauthor{\bsnm{Cl{\'e}men{\c c}on},~\bfnm{St{\'e}phan}\binits{S.}}
(\byear{2007}).
\btitle{Second-order properties of regeneration-based bootstrap for atomic
  Markov chains}.
\bjournal{TEST}
\bvolume{16}
\bpages{109--122}.
\bdoi{10.1007/s11749-006-0004-z}
\end{barticle}
\endbibitem

\bibitem{Bertail2006}
\begin{barticle}[author]
\bauthor{\bsnm{Bertail},~\bfnm{Patrice}\binits{P.}} \AND
  \bauthor{\bsnm{Cl{\'e}men{\c{c}}on},~\bfnm{St{\'e}phan}\binits{S.}}
(\byear{2006}).
\btitle{Regenerative block bootstrap for Markov chains}.
\bjournal{Bernoulli}
\bvolume{12}
\bpages{689--712}.
\end{barticle}
\endbibitem

\bibitem{Bickel1981}
\begin{barticle}[author]
\bauthor{\bsnm{Bickel},~\bfnm{Peter~J.}\binits{P.~J.}} \AND
  \bauthor{\bsnm{Freedman},~\bfnm{David~A.}\binits{D.~A.}}
(\byear{1981}).
\btitle{Some Asymptotic Theory for the Bootstrap}.
\bjournal{The Annals of Statistics}
\bvolume{9}.
\bdoi{10.2307/2240410}
\end{barticle}
\endbibitem

\bibitem{Billingsley1968}
\begin{bbook}[author]
\bauthor{\bsnm{Billingsley},~\bfnm{Patrick}\binits{P.}}
(\byear{1968}).
\btitle{Convergence of probability measures}.
\bpublisher{Wiley}.
\end{bbook}
\endbibitem

\bibitem{Bingham1987}
\begin{bbook}[author]
\bauthor{\bsnm{Bingham},~\bfnm{N.~H.}\binits{N.~H.}},
  \bauthor{\bsnm{Goldie},~\bfnm{C.~M.}\binits{C.~M.}} \AND
  \bauthor{\bsnm{Teugels},~\bfnm{J.~L.}\binits{J.~L.}}
(\byear{1987}).
\btitle{Regular variation}.
\bseries{Encyclopedia of mathematics and its applications 27}.
\bpublisher{Cambridge University Press}.
\end{bbook}
\endbibitem

\bibitem{Blackwell1945}
\begin{barticle}[author]
\bauthor{\bsnm{Blackwell},~\bfnm{David}\binits{D.}}
(\byear{1945}).
\btitle{{The existence of anormal chains}}.
\bjournal{Bulletin of the American Mathematical Society}
\bvolume{51}
\bpages{465 -- 468}.
\bdoi{bams/1183506994}
\end{barticle}
\endbibitem

\bibitem{CavaliereGeorgiev2020}
\begin{barticle}[author]
\bauthor{\bsnm{Cavaliere},~\bfnm{Giuseppe}\binits{G.}} \AND
  \bauthor{\bsnm{Georgiev},~\bfnm{Iliyan}\binits{I.}}
(\byear{2020}).
\btitle{Inference Under Random Limit Bootstrap Measures}.
\bjournal{Econometrica}
\bvolume{88}
\bpages{2547-2574}.
\end{barticle}
\endbibitem

\bibitem{Cavaliere2015}
\begin{barticle}[author]
\bauthor{\bsnm{Cavaliere},~\bfnm{Giuseppe}\binits{G.}},
  \bauthor{\bsnm{Politis},~\bfnm{Dimitris~N.}\binits{D.~N.}} \AND
  \bauthor{\bsnm{Rahbek},~\bfnm{Anders}\binits{A.}}
(\byear{2015}).
\btitle{Recent Developments in Bootstrap Methods for Dependent Data}.
\bjournal{Journal of Time Series Analysis 2015-mar 16 vol. 36 iss. 3}
\bvolume{36}.
\bdoi{10.1111/jtsa.12128}
\end{barticle}
\endbibitem

\bibitem{Chen1999}
\begin{barticle}[author]
\bauthor{\bsnm{Chen},~\bfnm{Xia}\binits{X.}}
(\byear{1999}).
\btitle{How Often Does a Harris Recurrent Markov Chain Recur?}
\bjournal{The Annals of Probability}
\bvolume{27}.
\bdoi{10.1214/aop/1022677449}
\end{barticle}
\endbibitem

\bibitem{Chen2000}
\begin{barticle}[author]
\bauthor{\bsnm{Chen},~\bfnm{Xia}\binits{X.}}
(\byear{2000}).
\btitle{On the limit laws of the second order for additive functionals of
  Harris recurrent Markov chains}.
\bjournal{Probability Theory and Related Fields}
\bvolume{116}.
\bdoi{10.1007/pl00008724}
\end{barticle}
\endbibitem

\bibitem{Deconinch2009}
\begin{barticle}[author]
\bauthor{\bsnm{{De Coninck}},~\bfnm{Jo{\"e}l}\binits{J.}},
  \bauthor{\bsnm{Dunlop},~\bfnm{Fran{\c c}ois}\binits{F.}} \AND
  \bauthor{\bsnm{Huillet},~\bfnm{Thierry}\binits{T.}}
(\byear{2009}).
\btitle{Random walk versus random line}.
\bjournal{Physica A: Statistical Mechanics and its Applications}
\bvolume{388}
\bpages{4034-4040}.
\bdoi{https://doi.org/10.1016/j.physa.2009.06.030}
\end{barticle}
\endbibitem

\bibitem{Csoergo2003}
\begin{barticle}[author]
\bauthor{\bsnm{Cs{\"o}rg{\H{o}}},~\bfnm{S{\'a}ndor}\binits{S.}} \AND
  \bauthor{\bsnm{Rosalsky},~\bfnm{Andrew}\binits{A.}}
(\byear{2003}).
\btitle{A survey of limit laws for bootstrapped sums}.
\bjournal{International Journal of Mathematics and Mathematical Sciences}
\bvolume{2003}
\bpages{2835--2861}.
\end{barticle}
\endbibitem

\bibitem{Somnat1996}
\begin{barticle}[author]
\bauthor{\bsnm{Datta},~\bfnm{Somnath}\binits{S.}}
(\byear{1996}).
\btitle{On asymptotic properties of bootstrap for AR(1) processes}.
\bjournal{Journal of Statistical Planning and Inference 1996-aug vol. 53 iss.
  3}
\bvolume{53}.
\bdoi{10.1016/0378-3758(95)00147-6}
\end{barticle}
\endbibitem

\bibitem{Somnat-1993}
\begin{barticle}[author]
\bauthor{\bsnm{Datta},~\bfnm{Somnath}\binits{S.}} \AND
  \bauthor{\bsnm{McCormick},~\bfnm{William~P.}\binits{W.~P.}}
(\byear{1993}).
\btitle{Regeneration-Based Bootstrap for Markov Chains}.
\bjournal{Canadian Journal of Statistics}
\bvolume{21}.
\bdoi{10.2307/3315810}
\end{barticle}
\endbibitem

\bibitem{DelFoff2002}
\begin{barticle}[author]
\bauthor{\bsnm{Delattre},~\bfnm{Sylvain}\binits{S.}} \AND
  \bauthor{\bsnm{Hoffmann},~\bfnm{Marc}\binits{M.}}
(\byear{2002}).
\btitle{{Asymptotic equivalence for a null recurrent diffusion}}.
\bjournal{Bernoulli}
\bvolume{8}
\bpages{139 -- 174}.
\end{barticle}
\endbibitem

\bibitem{Doeblin1940}
\begin{barticle}[author]
\bauthor{\bsnm{Doeblin},~\bfnm{W.}\binits{W.}}
(\byear{1940}).
\btitle{Éléments d'une théorie générale des chaînes simples constantes de
  Markoff}.
\bjournal{Annales scientifiques de l'École Normale Supérieure}
\bvolume{57}
\bpages{61-111}.
\end{barticle}
\endbibitem

\bibitem{markovChain2018}
\begin{bbook}[author]
\bauthor{\bsnm{Douc},~\bfnm{Randal}\binits{R.}},
  \bauthor{\bsnm{Moulines},~\bfnm{Eric}\binits{E.}},
  \bauthor{\bsnm{Priouret},~\bfnm{Pierre}\binits{P.}} \AND
  \bauthor{\bsnm{Soulier},~\bfnm{Philippe}\binits{P.}}
(\byear{2018}).
\btitle{Markov chains}.
\bseries{Springer Series in Operations Research and Financial Engineering}.
\bpublisher{Springer}.
\end{bbook}
\endbibitem

\bibitem{efron1979}
\begin{barticle}[author]
\bauthor{\bsnm{Efron},~\bfnm{B.}\binits{B.}}
(\byear{1979}).
\btitle{{Bootstrap Methods: Another Look at the Jackknife}}.
\bjournal{The Annals of Statistics}
\bvolume{7}
\bpages{1 -- 26}.
\bdoi{10.1214/aos/1176344552}
\end{barticle}
\endbibitem

\bibitem{FrankeKreissMammen2002}
\begin{barticle}[author]
\bauthor{\bsnm{Franke},~\bfnm{J{\"u}rgen}\binits{J.}},
  \bauthor{\bsnm{Kreiss},~\bfnm{Jens-Peter}\binits{J.-P.}} \AND
  \bauthor{\bsnm{Mammen},~\bfnm{Enno}\binits{E.}}
(\byear{2002}).
\btitle{{Bootstrap of kernel smoothing in nonlinear time series}}.
\bjournal{Bernoulli}
\bvolume{8}
\bpages{1 -- 37}.
\end{barticle}
\endbibitem

\bibitem{Gao2013}
\begin{barticle}[author]
\bauthor{\bsnm{Gao},~\bfnm{Jiti}\binits{J.}},
  \bauthor{\bsnm{Tj{\o}stheim},~\bfnm{Dag}\binits{D.}} \AND
  \bauthor{\bsnm{Yin},~\bfnm{Jiying}\binits{J.}}
(\byear{2013}).
\btitle{Estimation in threshold autoregressive models with a stationary and a
  unit root regime}.
\bjournal{Journal of Econometrics}
\bvolume{172}
\bpages{1--13}.
\end{barticle}
\endbibitem

\bibitem{Gut2013}
\begin{bbook}[author]
\bauthor{\bsnm{Gut},~\bfnm{Allan}\binits{A.}}
(\byear{2013}).
\btitle{Probability : a graduate course},
\bedition{2nd ed} ed.
\bseries{Springer texts in statistics}.
\bpublisher{Springer}.
\end{bbook}
\endbibitem

\bibitem{HALL1990108}
\begin{barticle}[author]
\bauthor{\bsnm{Hall},~\bfnm{Peter}\binits{P.}}
(\byear{1990}).
\btitle{On the relative performance of bootstrap and Edgeworth approximations
  of a distribution function}.
\bjournal{Journal of Multivariate Analysis}
\bvolume{35}
\bpages{108-129}.
\bdoi{https://doi.org/10.1016/0047-259X(90)90019-E}
\end{barticle}
\endbibitem

\bibitem{Horowitz2003}
\begin{barticle}[author]
\bauthor{\bsnm{Horowitz},~\bfnm{Joel~L.}\binits{J.~L.}}
(\byear{2003}).
\btitle{Bootstrap Methods for Markov Processes}.
\bjournal{Econometrica}
\bvolume{71}
\bpages{1049-1082}.
\bdoi{https://doi.org/10.1111/1468-0262.00439}
\end{barticle}
\endbibitem

\bibitem{Horowitz2019}
\begin{barticle}[author]
\bauthor{\bsnm{Horowitz},~\bfnm{Joel~L.}\binits{J.~L.}}
(\byear{2019}).
\btitle{Bootstrap Methods in Econometrics}.
\bjournal{Annual Review of Economics}
\bvolume{11}
\bpages{193-224}.
\bdoi{10.1146/annurev-economics-080218-025651}
\end{barticle}
\endbibitem

\bibitem{JeanJacod2003}
\begin{bbook}[author]
\bauthor{\bsnm{Jacod},~\bfnm{Jean}\binits{J.}} \AND
  \bauthor{\bsnm{Shiryaev},~\bfnm{Albert~N.}\binits{A.~N.}}
(\byear{2003}).
\btitle{Limit Theorems for Stochastic Processes},
\bedition{2} ed.
\bseries{Grundlehren der mathematischen Wissenschaften 288}.
\bpublisher{Springer-Verlag Berlin Heidelberg}.
\end{bbook}
\endbibitem

\bibitem{jainJamison1967}
\begin{barticle}[author]
\bauthor{\bsnm{Jain},~\bfnm{Naresh}\binits{N.}} \AND
  \bauthor{\bsnm{Jamison},~\bfnm{Benton}\binits{B.}}
(\byear{1967}).
\btitle{Contributions to Doeblin's theory of Markov processes}.
\bjournal{Zeitschrift f{\"u}r Wahrscheinlichkeitstheorie und Verwandte Gebiete}
\bvolume{8}
\bpages{19--40}.
\bdoi{10.1007/BF00533942}
\end{barticle}
\endbibitem

\bibitem{Kallianpur1954}
\begin{barticle}[author]
\bauthor{\bsnm{Kallianpur},~\bfnm{H.}\binits{H.} \bsuffix{G.~;~Robbins}}
(\byear{1954}).
\btitle{The sequence of sums of independent random variables}.
\bjournal{Duke Mathematical Journal}
\bvolume{21}.
\bdoi{10.1215/S0012-7094-54-02128-6}
\end{barticle}
\endbibitem

\bibitem{Tjostheim-2001}
\begin{barticle}[author]
\bauthor{\bsnm{Karlsen},~\bfnm{Hans~Arnfinn}\binits{H.~A.}} \AND
  \bauthor{\bsnm{Tjostheim},~\bfnm{Dag}\binits{D.}}
(\byear{2001}).
\btitle{Nonparametric estimation in null recurrent time series}.
\bjournal{The Annals of Statistics}
\bvolume{29}.
\bdoi{10.1214/aos/1009210546}
\end{barticle}
\endbibitem

\bibitem{Kasahara1984}
\begin{barticle}[author]
\bauthor{\bsnm{Kasahara},~\bfnm{Yuji}\binits{Y.}}
(\byear{1984}).
\btitle{Limit theorems for Lévy processes and Poisson point processes and
  their applications to Brownian excursions}.
\bjournal{J. Math. Kyoto Univ.}
\bvolume{24}
\bpages{521--538}.
\bdoi{10.1215/kjm/1250521278}
\end{barticle}
\endbibitem

\bibitem{Knight1989}
\begin{barticle}[author]
\bauthor{\bsnm{Knight},~\bfnm{Keith}\binits{K.}}
(\byear{1989}).
\btitle{{On the Bootstrap of the Sample Mean in the Infinite Variance Case}}.
\bjournal{The Annals of Statistics}
\bvolume{17}
\bpages{1168 -- 1175}.
\bdoi{10.1214/aos/1176347262}
\end{barticle}
\endbibitem

\bibitem{KreissLahiri2012}
\begin{bincollection}[author]
\bauthor{\bsnm{Kreiss},~\bfnm{Jens-Peter}\binits{J.-P.}} \AND
  \bauthor{\bsnm{Lahiri},~\bfnm{Soumendra~Nath}\binits{S.~N.}}
(\byear{2012}).
\btitle{1 - Bootstrap Methods for Time Series}.
In \bbooktitle{Time Series Analysis: Methods and Applications}.
\bseries{Handbook of Statistics}
\bvolume{30}
\bpages{3-26}.
\bpublisher{Elsevier}.
\bdoi{https://doi.org/10.1016/B978-0-444-53858-1.00001-6}
\end{bincollection}
\endbibitem

\bibitem{KreissPaparoditis2003}
\begin{barticle}[author]
\bauthor{\bsnm{Kreiss},~\bfnm{Jens-Peter}\binits{J.-P.}} \AND
  \bauthor{\bsnm{Paparoditis},~\bfnm{Efstathios}\binits{E.}}
(\byear{2003}).
\btitle{Autoregressive-Aided Periodogram Bootstrap for Time Series}.
\bjournal{The Annals of Statistics}
\bvolume{31}
\bpages{1923--1955}.
\end{barticle}
\endbibitem

\bibitem{KreissPaparoditis2011}
\begin{barticle}[author]
\bauthor{\bsnm{Kreiss},~\bfnm{Jens-Peter}\binits{J.-P.}} \AND
  \bauthor{\bsnm{Paparoditis},~\bfnm{Efstathios}\binits{E.}}
(\byear{2011}).
\btitle{Bootstrap methods for dependent data: A review}.
\bjournal{Journal of the Korean Statistical Society}
\bvolume{40}
\bpages{357-378}.
\bdoi{https://doi.org/10.1016/j.jkss.2011.08.009}
\end{barticle}
\endbibitem

\bibitem{KreissPaparoditisPolitis2011}
\begin{barticle}[author]
\bauthor{\bsnm{Kreiss},~\bfnm{Jens-Peter}\binits{J.-P.}},
  \bauthor{\bsnm{Paparoditis},~\bfnm{Efstathios}\binits{E.}} \AND
  \bauthor{\bsnm{Politis},~\bfnm{Dimitris~N.}\binits{D.~N.}}
(\byear{2011}).
\btitle{{On the range of validity of the autoregressive sieve bootstrap}}.
\bjournal{The Annals of Statistics}
\bvolume{39}
\bpages{2103 -- 2130}.
\bdoi{10.1214/11-AOS900}
\end{barticle}
\endbibitem

\bibitem{Kulperger1989}
\begin{barticle}[author]
\bauthor{\bsnm{Kulperger},~\bfnm{R.~J.}\binits{R.~J.}} \AND
  \bauthor{\bsnm{Rao},~\bfnm{B.~L. S.~Prakasa}\binits{B.~L. S.~P.}}
(\byear{1989}).
\btitle{Bootstrapping a Finite State Markov Chain}.
\bjournal{Sankhyā: The Indian Journal of Statistics, Series A (1961-2002)}
\bvolume{51}
\bpages{178--191}.
\end{barticle}
\endbibitem

\bibitem{Kunsch1989}
\begin{barticle}[author]
\bauthor{\bsnm{Kunsch},~\bfnm{Hans~R.}\binits{H.~R.}}
(\byear{1989}).
\btitle{{The Jackknife and the Bootstrap for General Stationary Observations}}.
\bjournal{The Annals of Statistics}
\bvolume{17}
\bpages{1217 -- 1241}.
\bdoi{10.1214/aos/1176347265}
\end{barticle}
\endbibitem

\bibitem{lahiri2003resampling}
\begin{bbook}[author]
\bauthor{\bsnm{Lahiri},~\bfnm{SN}\binits{S.}}
(\byear{2003}).
\btitle{Resampling methods for dependent data}.
\bpublisher{Springer Science \& Business Media}.
\end{bbook}
\endbibitem

\bibitem{LePage1981}
\begin{barticle}[author]
\bauthor{\bsnm{LePage},~\bfnm{Raoul}\binits{R.}},
  \bauthor{\bsnm{Woodroofe},~\bfnm{Michael}\binits{M.}} \AND
  \bauthor{\bsnm{Zinn},~\bfnm{Joel}\binits{J.}}
(\byear{1981}).
\btitle{{Convergence to a Stable Distribution Via Order Statistics}}.
\bjournal{The Annals of Probability}
\bvolume{9}
\bpages{624 -- 632}.
\bdoi{10.1214/aop/1176994367}
\end{barticle}
\endbibitem

\bibitem{MacKinnon2006}
\begin{barticle}[author]
\bauthor{\bsnm{MacKinnon},~\bfnm{James~G.}\binits{J.~G.}}
(\byear{2006}).
\btitle{Bootstrap Methods in Econometrics*}.
\bjournal{Economic Record}
\bvolume{82}
\bpages{S2-S18}.
\bdoi{https://doi.org/10.1111/j.1475-4932.2006.00328.x}
\end{barticle}
\endbibitem

\bibitem{Malinovskii1987}
\begin{barticle}[author]
\bauthor{\bsnm{Malinovskii},~\bfnm{V.~K.}\binits{V.~K.}}
(\byear{1987}).
\btitle{Limit Theorems for Harris Markov Chains, I}.
\bjournal{Theory of Probability and Its Applications}
\bvolume{31}.
\bdoi{10.1137/1131033}
\end{barticle}
\endbibitem

\bibitem{Meyn2009}
\begin{bbook}[author]
\bauthor{\bsnm{Meyn},~\bfnm{Sean}\binits{S.}},
  \bauthor{\bsnm{Tweedie},~\bfnm{Richard}\binits{R.}} \AND
  \bauthor{\bsnm{Glynn},~\bfnm{Peter}\binits{P.}}
(\byear{2009}).
\btitle{Markov chains and stochastic stability},
\bedition{2} ed.
\bseries{Cambridge Mathematical Library}.
\bpublisher{Cambridge University Press}.
\end{bbook}
\endbibitem

\bibitem{Myklebust-2012}
\begin{barticle}[author]
\bauthor{\bsnm{Myklebust},~\bfnm{Terje}\binits{T.}},
  \bauthor{\bsnm{Karlsen},~\bfnm{Hans~Arnfinn}\binits{H.~A.}} \AND
  \bauthor{\bsnm{Tjøstheim},~\bfnm{Dag}\binits{D.}}
(\byear{2012}).
\btitle{Null Recurrent Unit Root Process}.
\bjournal{Econometric Theory}
\bvolume{28}.
\bdoi{10.1017/S0266466611000119}
\end{barticle}
\endbibitem

\bibitem{Nummelin1978}
\begin{barticle}[author]
\bauthor{\bsnm{Nummelin},~\bfnm{E.}\binits{E.}}
(\byear{1978}).
\btitle{A splitting technique for Harris recurrent Markov chains}.
\bjournal{Probability Theory and Related Fields}
\bvolume{43}.
\bdoi{10.1007/bf00534764}
\end{barticle}
\endbibitem

\bibitem{Nummelin1984}
\begin{bbook}[author]
\bauthor{\bsnm{Nummelin},~\bfnm{Esa}\binits{E.}}
(\byear{1984}).
\btitle{General Irreducible Markov Chains and Non-Negative Operators}.
\bseries{Cambridge Tracts in Mathematics 83}.
\bpublisher{Cambridge University Press}.
\end{bbook}
\endbibitem

\bibitem{palm2008bootstrap}
\begin{barticle}[author]
\bauthor{\bsnm{Palm},~\bfnm{Franz~C}\binits{F.~C.}},
  \bauthor{\bsnm{Smeekes},~\bfnm{Stephan}\binits{S.}} \AND
  \bauthor{\bsnm{Urbain},~\bfnm{Jean-Pierre}\binits{J.-P.}}
(\byear{2008}).
\btitle{Bootstrap unit-root tests: comparison and extensions}.
\bjournal{Journal of Time Series Analysis}
\bvolume{29}
\bpages{371--401}.
\end{barticle}
\endbibitem

\bibitem{Paparoditis2002}
\begin{binbook}[author]
\bauthor{\bsnm{Paparoditis},~\bfnm{Efstathios}\binits{E.}}
(\byear{2002}).
\btitle{Frequency Domain Bootstrap for Time Series}
In \bbooktitle{Empirical Process Techniques for Dependent Data}
\bpages{365--381}.
\bpublisher{Birkh{\"a}user Boston}, \baddress{Boston, MA}.
\bdoi{10.1007/978-1-4612-0099-4_14}
\end{binbook}
\endbibitem

\bibitem{PaparoditisPolitis2000}
\begin{binbook}[author]
\bauthor{\bsnm{Paparoditis},~\bfnm{Efstathios}\binits{E.}} \AND
  \bauthor{\bsnm{Politis},~\bfnm{Dimitris~N.}\binits{D.~N.}}
(\byear{2000}).
\btitle{The Continuous-Path Block-Bootstrap}
In \bbooktitle{Papers in Honor of George Gregory Roussas}
\bpages{305--320}.
\bpublisher{De Gruyter}, \baddress{Berlin, Boston}.
\bdoi{doi:10.1515/9783110942002-021}
\end{binbook}
\endbibitem

\bibitem{PaparoditisPolitis2001a}
\begin{barticle}[author]
\bauthor{\bsnm{Paparoditis},~\bfnm{Efstathios}\binits{E.}} \AND
  \bauthor{\bsnm{Politis},~\bfnm{Dimitris~N.}\binits{D.~N.}}
(\byear{2001}).
\btitle{Tapered Block Bootstrap}.
\bjournal{Biometrika}
\bvolume{88}
\bpages{1105--1119}.
\end{barticle}
\endbibitem

\bibitem{PaparoditisPolitis2001}
\begin{barticle}[author]
\bauthor{\bsnm{Paparoditis},~\bfnm{Efstathios}\binits{E.}} \AND
  \bauthor{\bsnm{Politis},~\bfnm{Dimitris~N.}\binits{D.~N.}}
(\byear{2001}).
\btitle{A Markovian Local Resampling Scheme for Nonparametric Estimators in
  Time Series Analysis}.
\bjournal{Econometric Theory}
\bvolume{17}
\bpages{540--566}.
\end{barticle}
\endbibitem

\bibitem{PaparoditisPolitis2002a}
\begin{barticle}[author]
\bauthor{\bsnm{Paparoditis},~\bfnm{Efstathios}\binits{E.}} \AND
  \bauthor{\bsnm{Politis},~\bfnm{Dimitris~N.}\binits{D.~N.}}
(\byear{2002}).
\btitle{Local block bootstrap}.
\bjournal{Comptes Rendus Mathematique}
\bvolume{335}
\bpages{959-962}.
\bdoi{https://doi.org/10.1016/S1631-073X(02)02578-5}
\end{barticle}
\endbibitem

\bibitem{PaparoditisPolitis2002}
\begin{barticle}[author]
\bauthor{\bsnm{Paparoditis},~\bfnm{Efstathios}\binits{E.}} \AND
  \bauthor{\bsnm{Politis},~\bfnm{Dimitris~N.}\binits{D.~N.}}
(\byear{2002}).
\btitle{The local bootstrap for Markov processes}.
\bjournal{Journal of Statistical Planning and Inference}
\bvolume{108}
\bpages{301-328}.
\bnote{C.R. Rao 80th Birthday Felicitation Volume, Part II}.
\bdoi{https://doi.org/10.1016/S0378-3758(02)00315-4}
\end{barticle}
\endbibitem

\bibitem{politis1991circular}
\begin{bbook}[author]
\bauthor{\bsnm{Politis},~\bfnm{Dimitris~N}\binits{D.~N.}} \AND
  \bauthor{\bsnm{Romano},~\bfnm{Joseph~P}\binits{J.~P.}}
(\byear{1991}).
\btitle{A circular block-resampling procedure for stationary data}.
\bpublisher{Purdue University. Department of Statistics}.
\end{bbook}
\endbibitem

\bibitem{PolitisRomano1994}
\begin{barticle}[author]
\bauthor{\bsnm{Politis},~\bfnm{Dimitris~N.}\binits{D.~N.}} \AND
  \bauthor{\bsnm{Romano},~\bfnm{Joseph~P.}\binits{J.~P.}}
(\byear{1994}).
\btitle{The Stationary Bootstrap}.
\bjournal{Journal of the American Statistical Association}
\bvolume{89}
\bpages{1303--1313}.
\end{barticle}
\endbibitem

\bibitem{Pollard1984}
\begin{bbook}[author]
\bauthor{\bsnm{Pollard},~\bfnm{David}\binits{D.}}
(\byear{1984}).
\btitle{Convergence of stochastic processes}.
\bpublisher{Springer}.
\end{bbook}
\endbibitem

\bibitem{Rajarshi1990}
\begin{barticle}[author]
\bauthor{\bsnm{Rajarshi},~\bfnm{M.~B.}\binits{M.~B.}}
(\byear{1990}).
\btitle{Bootstrap in Markov-sequences based on estimates of transition
  density}.
\bjournal{Annals of the Institute of Statistical Mathematics}
\bvolume{42}
\bpages{253--268}.
\end{barticle}
\endbibitem

\bibitem{Revesz2005}
\begin{bbook}[author]
\bauthor{\bsnm{Revesz},~\bfnm{Pal}\binits{P.}}
(\byear{2005}).
\btitle{Random Walk in Random and Non-Random Environments},
\bedition{2nd ed} ed.
\bpublisher{World Scientific}.
\end{bbook}
\endbibitem

\bibitem{XiaoBing2023}
\begin{barticle}[author]
\bauthor{\bsnm{Tang},~\bfnm{Xiao-Song}\binits{X.-S.}},
  \bauthor{\bsnm{Huang},~\bfnm{Han-Bing}\binits{H.-B.}},
  \bauthor{\bsnm{Liu},~\bfnm{Xiong-Feng}\binits{X.-F.}},
  \bauthor{\bsnm{Li},~\bfnm{Dian-Qing}\binits{D.-Q.}} \AND
  \bauthor{\bsnm{Liu},~\bfnm{Yong}\binits{Y.}}
(\byear{2023}).
\btitle{Efficient Bayesian method for characterizing multiple soil parameters
  using parametric bootstrap}.
\bjournal{Computers and Geotechnics}
\bvolume{156}
\bpages{105296}.
\bdoi{https://doi.org/10.1016/j.compgeo.2023.105296}
\end{barticle}
\endbibitem

\bibitem{Zoubir1998}
\begin{barticle}[author]
\bauthor{\bsnm{Zoubir},~\bfnm{A.~M.}\binits{A.~M.}} \AND
  \bauthor{\bsnm{Boashash},~\bfnm{B.}\binits{B.}}
(\byear{1998}).
\btitle{The bootstrap and its application in signal processing}.
\bjournal{IEEE Signal Processing Magazine}
\bvolume{15}
\bpages{56-76}.
\bdoi{10.1109/79.647043}
\end{barticle}
\endbibitem

\bibitem{zoubirIskander2004}
\begin{bbook}[author]
\bauthor{\bsnm{Zoubir},~\bfnm{Abdelhak~M.}\binits{A.~M.}} \AND
  \bauthor{\bsnm{Iskander},~\bfnm{D.~Robert}\binits{D.~R.}}
(\byear{2004}).
\btitle{Bootstrap Techniques for Signal Processing}.
\bpublisher{Cambridge University Press}.
\bdoi{10.1017/CBO9780511536717}
\end{bbook}
\endbibitem

\end{thebibliography}


\end{document}